\documentclass[reqno]{amsart}
\usepackage{dsfont}
\usepackage{mathdots}
\usepackage{mathrsfs}
\usepackage{fancyhdr}
\usepackage{color}
\usepackage{enumitem}
\usepackage[pdftex]{hyperref}
\usepackage[sort&compress,numbers]{natbib}
\usepackage{curves}
\newtheorem{thm}{Theorem}[section]

\newtheorem{lem}[thm]{Lemma}
\newtheorem{prop}[thm]{Proposition}
\theoremstyle{definition}

\theoremstyle{remark}
\newtheorem{rem}[thm]{Remark}
\newtheorem*{acknow}{Acknowledgments}
\numberwithin{equation}{section}
\newcommand{\eps}{\varepsilon}

\renewcommand{\citet}[1]{\textsuperscript{\cite{#1}}}
\newcommand{\paren}[1]{\left(#1\right)}      
\newcommand{\set}[1]{\left\{#1\right\}}
\newcommand{\abs}[1]{\left\vert#1\right\vert}        

\newcommand{\Real}[1]{\mathbb{#1}}

\newcommand{\trace}[1]{\mathrm{tr}\left(#1\right)}

\renewcommand{\Im}{\mathrm{Im}}
\renewcommand{\Re}[1]{\mathrm{Re}\left\{#1\right\}}

\newcommand{\Det}[1]{\mathrm{det}\left(#1\right)}

\begin{document}

\title[Universality for products of random matrices]{Universality for products of random matrices I: Ginibre and truncated unitary cases}

\author{Dang-Zheng Liu} \address{Key Laboratory of Wu Wen-Tsun Mathematics, Chinese Academy of Sciences, School of Mathematical Sciences, University of Science and Technology of China, Hefei 230026, P.R.~China} \email{dzliu@ustc.edu.cn}

\author{Yanhui Wang} \address{Department of Mathematics, Harbin Institute of Technology, Harbin 150001, P. R. China} \email{yh\_wang@hit.edu.cn}

\date{\today}
\begin{abstract}
Recently, the joint  probability density functions of complex eigenvalues for products of   independent complex Ginibre  matrices
 have been explicitly derived as  determinantal point processes. We   express truncated series coming from   the   correlation kernels as multivariate integrals with singularity  and   investigate saddle point method for such a type of integrals. As an application, we prove that the eigenvalue correlation functions  have the same   scaling limits as those of the single complex Ginibre ensemble, both in the bulk and at the edge of the spectrum.  We also prove that the similar results  hold true for   products of independent truncated unitary matrices.   
\end{abstract}
\maketitle

\small
\tableofcontents
\normalsize

 \newpage
\section{Introduction and main results}

\subsection{Introduction}

Recently, there has been a   surge of interest   in   statistical properties  of eigenvalues for products of  independent random matrices with normal or more general iid entries in the literature, see \cite{burda2013free}, \cite{forrester2013lyapunov} and references therein for applications of such products. In particular, in a series of papers \cite{akemann2012universal,akemann2014universal,akemann2014permanental} Akemann and his coworkers showed that eigenvalues for products of matrices taken from   complex Ginibre ensembles and truncated unitary matrices form determinantal point processes; see also \cite{adhikari2013determinantal,ipsen2014weak} for mixed products of Ginibre matrices,  truncated unitary matrices and the inverses. These open up possibility of asymptotic analysis for local   eigenvalue statistics; see \cite{akemann2015suvey} for a recent survey. The goal of the present paper is to   investigate saddle point method for   a class  of multivariate  integrals with singularity. As applications,  in both  the complex Ginibre and truncated unitary cases, we  prove the pointwise  convergence of    one-point correlation functions   to some density with a single annulus support (see \cite{feinberg1997non} and \cite{guionnet2011single} for the  \textit{single ring phenomenon}) and also evaluate the bulk and edge scaling limits. A forthcoming paper \cite{wang2014} is devoted to studying global and local eigenvalue statistics for a mixed product of Ginibre matrices, truncated unitary matrices and their inverses.

  The principal subject of this paper is to study  two types of products of random matrices    defined as follows.

 \textbf{(i) Products of    complex Ginibre   matrices.}   It refers to the product of independent  induced complex Ginibre ensembles    $X^{(m)} = X_{m} X_{m-1} \cdots X_{1}$  where each $N\times N$ matrix  $X_k$ has the joint probability  density  proportional to
 \begin{equation}  \big(\Det{X^{*}_{k}X_{k}}\big)^{a_k} e^{ - \trace{X^{*}_{k}X_{k}}}. \label{inducedginibre} \end{equation}
 When all $a_k$ are non-negative integers, such a product $X^{(m)}$ of square matrices   can be realized as the product of rectangular matrices with   normal  entries where $a_k$ represents the difference between dimensions. See \cite{fischmann2012induced} for the induced Ginibre ensemble and \cite{ipsen2014weak} for the equivalence of the  two kinds of products.

  \textbf{(ii) Products of   truncated unitary   matrices}. It refers to the product of independent  truncated unitary matrices  $X^{(m)} = X_{m} X_{m-1} \cdots X_{1}$  where each $N\times N$ matrix  $X_k$ has the joint probability  density  proportional to
  \begin{equation}
    \big(\Det{X^{*}_{k}X_{k}}\big)^{a_k} \big(\Det{I_{N}-X^{*}_{k}X_{k}}\big)^{L_k-N}1_{\{  I_{N}-X^{*}_{k}X_{k}>0 \}}, \label{truncatedunitary}
  \end{equation}
  where all $L_{k} \geq N$. In the special case where all $a_k=0$, $X_{k}$ can be treated as the top left $N\times N$ sub-block $X_{k}$ of the unitary group $U(N+L_k)$ with its Haar measure (in this case if some $L_k<N$,  the density of the truncated matrix $X_{k}$ is different in form from \eqref{truncatedunitary} since the later is no longer applicable.), see e.g. \cite{akemann2014universal} and  the discussion in Appendix A therein.    More generally, when all $a_k$ are non-negative integers such a product $X^{(m)}$ of square matrices   can be realized as the product of truncated rectangular sub-block from unitary matrices  where $a_k$ represents the difference between dimensions, see \cite{ipsen2014weak}.      The  product has been studied in \cite{adhikari2013determinantal,akemann2014universal,ipsen2014weak} and  the joint eigenvalue density can be found therein. 

  A fundamental problem in random matrix  theory is to determine the limiting spectral measure  of the   empirical spectral distribution of a random matrix ensemble as the matrix size tends to infinity.  In the non-Hermitian case the famous \textit{circular law} for the complex Gaussian matrix was first shown by Ginibre \cite{ginibre1965statistical};  see also Mehta \cite{mehta1967random}.
  For the general   case when entries are iid complex variables the circular law was described by Girko \cite{girko1984circular}, and was rigorously proved by Bai \cite{bai1997circular}, G\"{o}tze and Tikhomirov \cite{gotze2010circular}, Pan and Zhou  \cite{pan2010circular},  and Tao and Vu \cite{tao2008random,tao2010random2}. In particular, Tao and Vu \cite{tao2010random2} established the law with the minimum assumption that the entry  has finite variance. More recently, the limit law for    products of $m$ independent complex square matrices was   originally considered by Burda, Janik and Waclaw \cite{burda2010spectrum} in  the Gaussian case, and in the iid case then by G\"{o}tze and Tikhomirov \cite{gotze2010asymptotic},  O'Rourke and Soshnikov \cite{o2011products},  and in the elliptic case by O'Rourke, Renfrew, Soshnikov and Vu \cite{o2015products}. As for the product of rectangular matrices, see  \cite{adhikari2013determinantal, burda2010eigenvalues, ipsen2014weak} for the Gaussian case and \cite{Tikhomirov2011Asymptotic,gotze2015Asymptotic} for  the  general case.


  Another fundamental problem in random matrix  theory is to prove the universality of local eigenvalue statistics.  In the non-Hermitian case (see \cite{khoruzhenko2011} for a detailed review), the scaling limits in the bulk and at the  edge are first computed in the complex Ginibre ensemble, see \cite{borodin2009ginibre}, \cite{forrester1999exact}, \cite{forrester2010log} and \cite{mehta1967random} for more details. Very recently, the same limits have been proven to be true  by  Fischmann etc. \cite{fischmann2012induced} in the induced complex Ginibre ensemble, and by Tao and Vu \cite{tao2015random} in the iid case where entries match  moments with the complex Ginibre ensemble to fourth order (Four moment theorem has been established therein). When turning to the product of random matrices, Akemann and  Burda \cite{akemann2012universal} calculated scaling limits in the bulk and at the edge  for products of $m$ independent complex Ginibre ensembles with each $a_k=0$ in \eqref{inducedginibre}.  In particular, a new universal limit at the origin is found, see \cite{osborn2004universal} for $m=2$ and \cite{akemann2012universal} for any fixed $m$.

  Another situation is the truncation of random unitary matrices described   in \cite{zyczkowski2000truncations}, see \cite{dong2012circular,jiang2009approximation,petz2005large} for the limiting spectral measure and \cite{zyczkowski2000truncations} for local scaling limits. Products of such   truncated unitary matrices  have been  investigated   as determinantal processes, based on which the  limit law for eigenvalues is obtained, see e.g.  \cite{adhikari2013determinantal,akemann2014universal,ipsen2014weak}.   Also the scaling limits in the bulk and at the edge  have been evaluated in \cite{akemann2014universal} when  each $a_k=0$  and all the truncations $L_k$ are equal  in \eqref{truncatedunitary}; the case of  $L_{k} < N$ is also treated therein.

\subsection{Main results}

 The joint density function of eigenvalues for the product of  induced complex Ginibre ensembles $X^{(m)}=X_mX_{m-1}\cdots X_1$ with $X_k$ distributed according to \eqref{inducedginibre} reads off \cite{adhikari2013determinantal,akemann2012universal,akemann2014universal}
    \begin{equation}
      P_{N}(z_1,\ldots,z_N) = C \prod_{l=1}^{N} w_{m}(z_{l}) \prod_{1 \leq j<k  \leq N} \abs{z_{k} - z_{j}}^{2},
    \end{equation}
    with respect to the Lebesgue measure on $\Real{C}^{N}$    denoted by $\mu(d z_1)\times \cdots \times \mu(d z_N)$.
Here the constant
    \begin{equation*}
      C= \frac{1}{N!} \prod_{l=0}^{N-1} \prod_{k=1}^{m} \frac{\Gamma(a_{k} + 1)}{\Gamma(a_{k} + l + 1)},
    \end{equation*}
    and the weight function can be expressed in terms of the Meijer G-functions as
    \begin{equation}
      w_{m}(z) = \frac{1}{\pi} \prod_{k=1}^{m} \frac{1}{\Gamma(a_{k}+1)} G_{0,m}^{m,0}(\abs{z}^{2} | a_{1}, \cdots, a_{m}).
    \end{equation}
The weight function $w_{m}(z)$ can also be given by the high-dimensional integral
    \begin{equation}\label{productWeightFunctionSpecifiedbyGinibre}
      \begin{split}
        w_{m}(z)&=\abs{z}^{2a_{m}} \frac{2^{m-1}}{\pi} \prod_{k=1}^{m} \frac{1}{\Gamma(a_{k}+1)} \int_{\Real{R}_{+}^{m-1}} e^{-p(r)} q(r) d^{m-1}r,
      \end{split}
    \end{equation}
    where
    \begin{equation}\label{pforWeightFunction}
      p(r) = p(r_{1}, \cdots, r_{m-1}) = \sum_{k=1}^{m-1} r_{k}^{2} + \frac{\abs{z}^{2}}{r_{1}^{2} \cdots r_{m-1}^{2}},
    \end{equation}
    and
    \begin{equation}\label{qforWeightFunction}
      q(r) = q(r_{1}, \cdots, r_{m-1}) = \prod_{k=1}^{m-1} r_{k}^{2(a_{k}-a_{m})-1},
    \end{equation}
    cf.  \cite{adhikari2013determinantal,akemann2014permanental}.
    This integral representation  proves to be suitable  for asymptotic analysis as $N$ goes to infinity.

As a determinantal  point process the kernel function reads off
  \begin{equation}\label{kernelFunctionofProductInducedGaussianWeight}
    K_{N}(z, z') = \sqrt{w_{m}(z)w_{m}(\overline{z'})} \sum_{l=0}^{N-1} \prod_{k=1}^{m} \frac{\Gamma(a_{k} + 1)}{\Gamma(a_{k} + l + 1)} (z \overline{z'})^{l}.
  \end{equation}
  Actually, noticing the orthogonal relations \cite{akemann2014permanental}
  \begin{equation}
    \int_{\Real{C}} z^{j} \bar{z}^{l} w_{m}(z) \mu(dz) = \delta_{j,l} \prod_{k=1}^{m} \frac{\Gamma(a_{k}+l+1)}{\Gamma(a_{k}+1)},
  \end{equation}
 and using Dyson's integral formula \cite{forrester2010log}, we have
  \begin{equation*}
    P_{N}(z_1,\ldots,z_N) = \frac{1}{N!} \Det{K_{N}(z_{j}, z_{k})}_{1 \leq j, k \leq N}
  \end{equation*}
  and further get  the $n$-point  correlation functions
  \begin{align}
    R_{N,n}(z_{1}, \cdots, z_{n}) :&= \frac{N!}{(N-n)!}\int_{\Real{C}}\cdots \int_{\Real{C}} P_N(z) \,\mu(d z_{n+1})\times \cdots \times \mu(d z_N)\nonumber\\
    &=\Det{K_{N}(z_{j}, z_{k})}_{1 \leq j, k \leq n}.\label{correlationfunction}
  \end{align}


One of our main results concerns local universality for products of independent Ginibre matrices with fixed parameters, see Theorem \ref{limitdensityfixed} in Section \ref{ginibrefixedsection}  for the pointwise convergence  of one-point correlation functions. In the case where all $a_k=0$ those have been  investigated by Akemann and  Burda \cite{akemann2012universal}.
\begin{thm} \label{univeralitytheoremfixed} For fixed $a_{1}, \ldots, a_{m} > -1$, let   $\rho = \frac{1}{\sqrt{m}} \abs{u}^{-(m-1)/m}$ and
     introduce rescaling variables      \begin{equation*}
      z_{k} = u + \frac{v_{k}}{\rho \sqrt{N}}, \quad k = 1, \cdots, n.
    \end{equation*}
 Then for $0 < \abs{u} < 1$,
         \begin{multline} \label{scalingLimitBulk}
          \lim_{N \to \infty} N^{n(m-1)} \rho^{-2n} R_{N,n}\big(N^{m/2}z_1,\ldots, N^{m/2}z_n\big) \\= \Det{\frac{1}{\pi}e^{-\frac{1}{2} (\abs{v_{k}}^{2} + \abs{v_{j}}^{2} - 2v_{k}\overline{v_{j}})}}_{1\leq k, j \leq n}
         \end{multline}
           while for $u= e^{i\phi}$ with $0\leq \phi<2\pi$,
        \begin{multline}
     \label{scalingLimitEdge}\lim_{N \to \infty} N^{n(m-1)}  \rho^{-2n} R_{N,n}\big(N^{m/2}z_1,\ldots, N^{m/2}z_n\big)
             \\=  \Det{\frac{1}{2\pi}e^{-\frac{1}{2} (\abs{v_{k}}^{2} + \abs{v_{j}}^{2} - 2v_{k}\overline{v_{j}})} \mathrm{erfc}\big(\frac{e^{i\phi}\overline{v_{j}} + e^{-i\phi}v_{k}}{\sqrt{2}}\big)}_{1\leq k, j \leq n}.
        \end{multline}
   Moreover, both hold true uniformly for  $v_{1}, \cdots, v_{n}$ in any compact subset of $\mathbb{C}$.
  \end{thm}

The next two of our main results concern  limiting  eigenvalue density and  local universality for products of independent Ginibre matrices with varying parameters.
  In order to use easily,  for $r_{\mathrm{in}}>0$ we introduce  a  function    \begin{equation}  \label{characteristicfuncion1} \chi_{\set{u:r_{\textrm{in}}<\abs{u} < r_{\textrm{out}}}}(z)=\begin{cases}
          1  , & r_{\mathrm{in}}<\abs{z}<r_{\mathrm{out}} \\
          1/2, &  \abs{z}=r_{\mathrm{in}}, r_{\mathrm{out}}\\
        0, &  otherwise,
        \end{cases} \end{equation}
        while for  $r_{\mathrm{in}}= 0$ we define
        \begin{equation}  \label{characteristicfuncion2} \chi_{\set{u:\abs{u} < r_{\textrm{out}}}}(z)=\begin{cases}
          1  , &  \abs{z}<r_{\mathrm{out}} \\
          1/2, &  \abs{z}=r_{\mathrm{out}}\\
        0, &  otherwise.
        \end{cases} \end{equation}

        Note that in the present paper we are not concerned with the asymptotic behavior at zero.  As a matter of fact, when $m>1$ the situation  is different from the single Ginibre ensemble, since in the former case the origin might be the hard edge of the spectrum (cf. Theorem \ref{limitdensityfixed}  in Section \ref{ginibrefixedsection} below). Besides, the scaling limits at the origin  have been obtained in \cite{akemann2012universal,akemann2014permanental}. Viewed from this perspective, the evaluation of the functions  at zero  in \eqref{characteristicfuncion1} and \eqref{characteristicfuncion2} is insignificant, cf. Theorem \ref{limitdensityvarying} below.

\begin{thm} \label{limitdensityvarying}
     For  the weight function $w_{m}(z)$   given by (\ref{productWeightFunctionSpecifiedbyGinibre}), suppose that $a_{1} = \delta_{1} N, \cdots, a_{m} = \delta_{m} N$  with $\delta_1,\ldots,\delta_m\geq 0$. For $z\neq 0$, let $\xi_{m}(z)$ be the largest real root of algebraic  equation in $x$
      \begin{equation}
        \prod_{k=1}^{m} (\delta_{k} - \delta_{m} + \abs{z}^{2}x) - \abs{z}^{2} = 0.
      \end{equation} Then  the limiting eigenvalue density
      \begin{equation}
        \begin{split}
          R_{1}(z) &:= \lim_{N \to \infty} N^{m-1}R_{N,1}(N^{m/2}z)\\
          &= \frac{1}{\pi \abs{z}^{2}} \frac{1}{ \sum_{k=1}^{m} \frac{1}{\delta_{k} - \delta_{m} + \abs{z}^{2} \xi_{m}(z)}} \chi_{\set{u:\sqrt{\delta_{1} \cdots \delta_{m}} < \abs{u} < \sqrt{(1+\delta_{1}) \cdots (1+\delta_{m})}}}(z)
        \end{split}
      \end{equation}
     holds true for any complex $z\neq 0$.
    \end{thm}

\begin{thm} \label{univeralitytheoremvarying}  With the same notation as in Theorem \ref{limitdensityvarying}, let
      \begin{equation}
        \rho = \frac{1}{\abs{u} \sqrt{\sum_{k=1}^{m} \frac{1}{\delta_{k} - \delta_{m} + \abs{u}^{2} \xi_{m}(u)}}},
      \end{equation}
            and introduce rescaling variables
      \begin{equation}
        z_{k} = u + \frac{v_{k}}{\rho \sqrt{N}}, \quad k=1, \cdots, n.
      \end{equation}
      Then  the following hold true uniformly for  $v_{1}, \cdots, v_{n}$ in any compact subset of $\mathbb{C}$.
      \begin{enumerate} [label = $\mathrm{(\arabic*)}$]
        \item \textbf{Bulk limit}.   \label{vScalingLimitBulk} For  $\sqrt{\delta_{1} \cdots \delta_{m}} < \abs{u} < \sqrt{(1+\delta_{1}) \cdots (1+\delta_{m})}$,
          \begin{multline}
            \lim_{N \to \infty} N^{n(m-1)} \rho^{-2n} R_{N,n}\big(N^{m/2}z_1, \ldots,N^{m/2}z_n\big) \\= \Det{\frac{1}{\pi}e^{-\frac{1}{2} (\abs{v_{k}}^{2} + \abs{v_{j}}^{2} - 2v_{k}\overline{v_{j}})}}_{1\leq k, j \leq n}.
          \end{multline}
        \item  \textbf{Inner edge}. For  $\delta_{1}, \ldots, \delta_{m}>0$ and $u = \sqrt{\delta_{1} \cdots \delta_{m}}e^{i\phi}$ with $0\leq \phi<2\pi$,
          \begin{multline}
             \lim_{N \to \infty}N^{n(m-1)} \rho^{-2n} R_{N,n}\big(N^{m/2}z_1, \ldots,N^{m/2}z_n\big) \\
               = \Det{\frac{1}{2\pi}e^{-\frac{1}{2} (\abs{v_{k}}^{2} + \abs{v_{j}}^{2} - 2v_{k}\overline{v_{j}})} \mathrm{erfc}\big(-\frac{e^{i\phi}\overline{v_{j}} + e^{-i\phi}v_{k} }{\sqrt{2} }\big)}_{1\leq k, j \leq n}.
                     \end{multline}
        \item  \textbf{Outer   edge}. For  $u = \sqrt{(1+\delta_{1}) \cdots (1+\delta_{m})} e^{i\phi}$ with $0\leq \phi<2\pi$,
          \begin{multline}
           \lim_{N \to \infty} N^{n(m-1)} \rho^{-2n} R_{N,n}\big(N^{m/2}z_1, \ldots,N^{m/2}z_n\big) \\
               = \Det{\frac{1}{2\pi}e^{-\frac{1}{2} (\abs{v_{k}}^{2} + \abs{v_{j}}^{2} - 2v_{k}\overline{v_{j}})} \mathrm{erfc}
              \big(\frac{e^{i\phi}\overline{v_{j}} + e^{-i\phi}v_{k} }{\sqrt{2} }\big)
               }_{1\leq k, j \leq n}.
                   \end{multline}
      \end{enumerate}
    \end{thm}

Similar results hold true for products of truncated unitary matrices, see Theorem \ref{limitdensityTr} for limiting eigenvalue density and Theorem \ref{univeralitytheoremTr} for local universality in Section \ref{truncatedproductsection}.
\begin{rem}
      We claim that the limiting eigenvalue density $R_{1}(z)$ given in Theorem \ref{limitdensityvarying} is  symmetric with respect to the parameters $\delta_{1}$, $\cdots$, $\delta_{m}$. Actually, for $z\neq 0$ and for given $j<m$, comparing the root $\xi_{j}(z)$  of the equation in $x$
      \begin{equation}
        \prod_{k=1}^{m} (\delta_{k} - \delta_{j} + \abs{z}^{2}x) - \abs{z}^{2} = 0
      \end{equation}
       and the root $\xi_{m}(z)$  of the equation
      \begin{equation}
        \prod_{k=1}^{m} (\delta_{k} - \delta_{m} + \abs{z}^{2}x) - \abs{z}^{2} = 0,
      \end{equation}
      there is an obvious relation
      \begin{equation}
        \xi_{j}(z) = \frac{\delta_{j} - \delta_{m}}{\abs{z}^{2}} + \xi_{m}(z).
      \end{equation}
    If $\xi_{m}(z)$ is chosen as the largest real root, then  $\xi_{j}(z)$ is the corresponding largest real root. In this case we know that the factor in the density function
      \begin{equation}
        \sum_{k=1}^{m} \frac{1}{\delta_{k} - \delta_{j} + \abs{z}^{2} \xi_{j}(z)}
      \end{equation}
      is independent of $j$.
    \end{rem}

The paper is organized as follows. Section \ref{saddlesection} is devoted to saddle point method for a class of multivariate principal value integrals. Two   main theorems are presented and will be used to tackle with  the large $N$ asymptotics for correlation functions in next sections.    In Section \ref{ginibrefixedsection}, the pointwise convergence of one-point correlation functions is proven and scaling limits in the bulk and at the edge are evaluated for products of complex Ginibre ensembles with fixed parameters $a_k$.  The similar results are proven to hold true for    products of complex Ginibre ensembles with  varying  parameters   and  for products of truncated unitary matrices respectively in Sections \ref{varyingginibre} and   \ref{truncatedproductsection}.  Appendix \ref{appendixa} gives explicit expression of the determinant and the inverse for   one class of matrices, which come up repeatedly in the present paper and also in the similar calculations of relevant problems.

\section{Asymptotics for multivariate integrals with singularity}\label{saddlesection}
 To prove the bulk and edge universality, it is sufficient for us to analyse asymptotic behaviour of  both the weight function and the finite sum in \eqref{kernelFunctionofProductInducedGaussianWeight}  in detail.  Laplace's approximation (see, e.g. \cite[p.495]{wong2001asymptotic}) can be used to evaluate asymptotics of the weight function, but it does not work for the finite sum part. Our aim is to tackle  the difficulty and develop  a general method
 of asymptotically evaluating  the $N$-term sum as $N\rightarrow \infty$.  In fact, the finite sum appears ubiquitously as a factor of the correlation kernel in random matrix theory, especially in the determinant  point processes associated with complex eigenvalues.
We will first express the finite sum as a multivariate integral with singularity  (see Section \ref{repfinitesum} below)  and then do  asymptotic analysis.
   In this section we devote ourselves to  steepest decent method for such a type of multivariate integrals with singularity and prove the following two theorems.  For convenience, we just consider the integrals with singularity at zero since  the general case  can be obtained with a translation.

The first result will play an important role in dealing with the universality problem  in the bulk  for complex eigenvalues of random matrices.
  \begin{thm}\label{laplaceMethod1} Let  $(\theta;t) = (\theta_{1}, \cdots, \theta_{m}; t)$ and $D = \gamma_{1} \times \cdots \times \gamma_{m+1}\subset T$, consider the integral   \begin{equation}
      I(\lambda) = \mathrm{P.V.} \int_{D} e^{-\lambda p(\theta; t)-\sqrt{\lambda}f(\theta;t)} \frac{q(\theta;t)}{t} d^{m}\theta \,dt,
    \end{equation}
in which $p(\theta; t), f(\theta;t)$ and $q(\theta; t)$ are independent  of the positive parameter $\lambda$, single-valued and holomorphic in a  domain $T \subset \Real{C}^{m+1}$.   Here each   $\gamma_{k}$   is  a smooth simple curve and  is independent of $\lambda$. Suppose that  $(\theta;t) = (0;0)$ is an interior point of $D$  such that the followings hold:
    \begin{enumerate}[label = $\mathrm{(\roman*)}$]
      \item\label{condition1} The first order partial derivatives
        \begin{equation}
          \alpha:= \frac{\partial p(0;0)}{\partial t} \neq 0,  \qquad   \eta := \nabla_{\theta}f(0;0),
        \end{equation}
        and
        \begin{equation}
          \nabla_{\theta} p(0; 0):= \paren{\frac{\partial p(0; 0)}{\partial \theta_{1}}, \cdots, \frac{\partial p(0; 0)}{\partial \theta_{m}}} = 0.
        \end{equation}
      \item\label{condition2} For the Hessian matrix $A$ of the first $m$ variables,
        \begin{equation}
          \Det{A} = \Det{\frac{\partial^{2} {p(0;0)}}{\partial \theta_{k} \partial \theta_{j}}}_{1 \leq k, j \leq m} \neq 0.
        \end{equation}
        \item\label{condition2.5} \label{condition2.5} There exists   $\lambda_0>0$ such that for every  $\lambda\geq \lambda_0$, $I(\lambda)$ converges absolutely throughout its range except for one neighbourhood of $(0;0)$.

      \item\label{condition3}   For every given $\varepsilon>0$ let  $D_{\eps} = \set{(\theta;t) \in D| \max\{\abs{\theta_1}, \ldots,\abs{\theta_m},\abs{t}\} < \varepsilon},$ there exists $\lambda_1>0$ such that
        $\varrho(\varepsilon,\lambda_1)>0$ where  \begin{multline}\varrho(\varepsilon,\lambda_1) =\\
         \inf\left\{\Re{p(\theta;t)-p(0;0)+\frac{f(\theta;t)-f(0;0)}{\sqrt{\lambda}}}: (\theta;t)\in D\backslash D_{\varepsilon}, \lambda\geq \lambda_1\right\}.\end{multline}

      \item\label{condition4} $q(0;0) \neq 0$.
    \end{enumerate}
    Then as $\lambda \rightarrow \infty$ we have
    \begin{equation}
      I(\lambda) \sim i\pi  (2\pi/\lambda)^{m/2} q(0;0) e^{-\lambda p(0;0)-\sqrt{\lambda} f(0;0)} e^{\frac{1}{2} \eta A^{-1} \eta^{T}}\frac{\epsilon(\alpha,A)}{\sqrt{\det{A}}},
    \end{equation}
    where $\epsilon(\alpha,A)$ equals to 1 or -1   depending  on the argument  of $\alpha$ and the choice of the sign of $\sqrt{\det A}$.
  \end{thm}
  \begin{proof}  We proceed in two steps.

  \textbf{Step 1:}   special case where $f(\theta;t)=0$. First, the integral $I(\lambda)$  can be subdivided into two parts
    \begin{equation}
      I(\lambda) = \mathrm{P.V.} \int_{D_{\varepsilon}} (\cdot)+\mathrm{P.V.}  \int_{D \backslash D_{\varepsilon}} (\cdot):= I_{1}(\lambda) + I_{2}(\lambda).
    \end{equation}
    By Condition \ref{condition3}, for any given $\varepsilon > 0$  there exists a positive number $c$ such that $\Re{p(\theta;t)} \geq c + \Re{p(0;0)}$ for any $(\theta;t) \in D \backslash D_{\varepsilon}$. Therefore, by Condition \ref{condition2.5} and Condition \ref{condition3}
    \begin{align}\label{laplaceMethod1_Decay}
      \abs{I_{2}(\lambda)} &=\abs{\mathrm{P.V.} \int e^{-(\lambda-\lambda_0) p(\theta; t)} e^{-\lambda_0 p(\theta; t)} \frac{q(\theta;t)}{t} d^{m}\theta \,dt}\nonumber\\
      &\leq K e^{-(\lambda-\lambda_{0})(\Re{p(0;0)} + c)}
    \end{align}
    for some constant $K > 0$.

    If we set $$\beta=\left(\frac{\partial^{2} {p(0;0)}}{\partial \theta_{1} \partial t},\ldots,\frac{\partial^{2} {p(0;0)}}{\partial \theta_{m} \partial t}\right),$$ then Taylor expansion of $p(\theta; t)$ at $(0;0)$ gives
    \begin{equation}
      \begin{split}
        p(\theta;t&) = p(0;0) + \alpha t + \frac{1}{2} \theta A \theta^{T} + \frac{1}{2} \beta \theta^{T} t + \cdots \\
        &:= p(0;0) + \alpha t(1+S(\theta; t)) + \frac{1}{2} \theta (A+R) \theta^{T}.\label{pexpansion}
      \end{split}
    \end{equation}
    Here $R= (R_{kj}(\theta))$ is a symmetric matrix whose entries $R_{kj}(\theta)$ are power series of $\theta$ with $R_{kj}(0) = 0$, and $S(\theta;t)$ is a power series of $(\theta; t)$ with $S(0;0) = 0$.

    By Condition \ref{condition2}, we can choose $\varepsilon$ small enough such that the matrix $A+R$ is non-degenerate in the domain $D_{\varepsilon}$.
    Since $A+R$ is symmetric and non-degenerate, there exists a non-degenerate matrix $Q$ such that $A+R = QQ^{T}$. Take  change of variables $(\theta; t) \to (\theta';s)$ with $\theta' = \theta Q$ and $s = \alpha t(1+S(\theta;t))$,  then for sufficiently small $\varepsilon$ we have $\theta=\theta(\theta')$ and $t=t(s,\theta')$. Moreover, under the  mapping  if letting  $\Omega$ be the image of $ \{\theta \in \gamma_{1} \times \cdots \times \gamma_{m}|\max\{\abs{\theta_1}, \ldots,\abs{\theta_m}\} < \varepsilon\}$ and for given $\theta'$ letting $\gamma'(\theta')$ be  the image of $ \{t \in \gamma_{m+1}|\abs{t}< \varepsilon\}$, then
    \begin{equation}
      \begin{split}
        I_{1}(\lambda)
        &= e^{-\lambda p(0;0)} \int_{\Omega} e^{- \frac{1}{2} \lambda \theta' \theta'^{T}} d^{m}\theta' \paren{\mathrm{P. V.} \int_{\gamma'(\theta')} e^{-\lambda s} g(\theta';s) \frac{1}{s} ds}, \label{I1after}
      \end{split}
    \end{equation}
    where
    \begin{equation}\label{ffunction}
      g(\theta';s) = q(\theta;t) \Det{\frac{\partial \theta_{k}'}{\partial \theta_{j}}}^{-1}
      \frac{1+S(\theta;t)}
      {
      1+S(\theta;t)+t
      \frac{\partial S(\theta;t)}{\partial t}
      }.
    \end{equation}
     We emphasize that $\gamma'(\theta')$ is a smooth simple curve passing through $0$.

     On the other hand, let $t_{\pm}$ be the two endpoints of the curve  $ \{t \in \gamma_{m+1}|\abs{t}< \varepsilon\}$,
    by Condition \ref{condition3} there exists  $\eta_1>0$ such that $\Re{p(\theta;t_{\pm})-p(0;0)} \geq \eta_1$ for any $(\theta;t_{\pm}) \in D_{\varepsilon}$. Therefore, for given $\theta'$ if letting $s_{\pm}(\theta')$ be the image of $t_{\pm}$ under the mapping, there exists $\eta_2>0$ such that
    $\Re{\frac{1}{2} \theta'^{T} \theta' + s_{\pm}(\theta')} \geq  \eta_2$ for any $\theta' \in \Omega$.
    This shows that we can choose a simple closed curve  $\widetilde{\gamma} (\theta')\supset \gamma'(\theta')$ such that
    \begin{equation}
\Re{\frac{1}{2} \theta'^{T} \theta' + s} \geq  \eta_2 \label{positiveeta2}
\end{equation}
 holds for any $\theta'\in \Omega$ and $s\in \widetilde{\gamma} (\theta')\backslash \gamma'(\theta')$.

   Like in the estimate of \eqref{laplaceMethod1_Decay},  we know from \eqref{positiveeta2} that   substitution of $\gamma'(\theta')$ with $\widetilde{\gamma} (\theta') $ in \eqref{I1after} does not affect the leading order of $I_{1}(\lambda)$.
    Notice  \begin{equation*}
    \mathrm{P. V.} \int_{\widetilde{\gamma} (\theta')} e^{-\lambda s} g(\theta';s) \frac{1}{s} ds=\mathrm{P. V.} \int_{\widetilde{\gamma} } e^{-\lambda s} g(\theta';s) \frac{1}{s} ds,
    \end{equation*}
   where   $\widetilde{\gamma} $ is a smooth closed simple curve and is  independent of $\theta'$, we thus get
    \begin{equation}\label{laplaceMethod1_Domain}
      \begin{split}
        I_{1}(\lambda) &\sim e^{-\lambda p(0;0)} \int_{\Omega} e^{- \frac{1}{2} \lambda \theta'^{T} \theta'} d^{m}\theta' \, \mathrm{P. V.} \int_{\widetilde{\gamma}} e^{-\lambda s} g(\theta';s) \frac{1}{s} ds \\
      &  \sim e^{-\lambda p(0;0)} \int_{\Real{R}^{m}} e^{- \frac{1}{2} \lambda \theta'^{T} \theta'} d^{m}\theta'  \,\mathrm{P. V.} \int_{\widetilde{\gamma}} e^{-\lambda s} g(0;s) \frac{1}{s} ds \\
        &= i\pi  (2\pi/\lambda)^{m/2} q(0;0) e^{-\lambda p(0;0)} \frac{\epsilon(\alpha,A)}{\sqrt{\det{A}}}
      \end{split}
    \end{equation}
      where we have used  \eqref{ffunction} and the principle value integral
       \begin{equation*}
 \mathrm{P. V.} \int_{\widetilde{\gamma} } e^{-\lambda s} g(0;s) \frac{1}{s} ds=\pm i\pi g(0;0)
    \end{equation*}  by the Sokhotskyi-Plemelj formula \cite[Sect.14.1]{henrici1993applied}. Here the choice of $\pm$ depends on the orientation of integration path.  The sign of $\epsilon(\alpha,A)$ relies on both
    the orientation of integration path and the sign of $\sqrt{\det{A}}$.

    Combination of   \eqref{laplaceMethod1_Decay} and \eqref{laplaceMethod1_Domain} completes the proof of Theorem \ref{laplaceMethod1} in this case.

    \textbf{Step 2}: general case.  

    With \eqref{pexpansion} in mind the Taylor expansion of $f(\theta;t)$ at $(0;0)$ shows
     \begin{equation}
          f(\theta;t)-f(0;0)  =   \eta \theta^{T} +
         \frac{1}{2} \theta  \widetilde{R}(\theta)  \theta^{T}   +   t \widetilde{S}(\theta;t),
       \end{equation}
    where $\widetilde{R}(\theta)$ is a symmetric matrix whose entries are power series of $\theta$, and $\widetilde{S}(\theta;t)$ is a power series of $(\theta;t)$.

%

    Notice  Conditions \ref{condition2},  \ref{condition2.5} and   \ref{condition3}, following the approach in the special case and rescaling $\theta$ and $t$ respectively by $\frac{1}{\sqrt{\lambda}}$ and $\frac{1}{ \lambda}$ near zero, the desired result could be easily derived.
  \end{proof}

The second result will be used to tackle the universality problem  at the edge  for complex eigenvalues of random matrices.
  \begin{thm}\label{steepestDecentonEdge}
   Let  $t = (t_{1}, \cdots, t_{m}, t_{m+1})$ and $D = \gamma_{1} \times \cdots \times \gamma_{m+1}\subset T$, consider
    \begin{equation}
      I(\lambda) = \mathrm{P.V.} \int_{D} e^{-\lambda p(t)-\sqrt{\lambda} f(t)} \frac{q(t)}{t_{m+1}} d^{m+1}t,
    \end{equation}
    in which $p(t), f(t)$ and $q(t)$ are   single-valued and holomorphic in a  domain $T \subset \Real{C}^{m+1}$. Write $\eta: = (\widetilde{\eta},\eta_{m+1}) = \nabla f(0)$ where $\widetilde{\eta}$ is an $m$-dimensional vector.
    Suppose that  $t = 0$ is an interior point of $D$  such that the followings hold:
    \begin{enumerate}[label = $\mathrm{(\roman*)}$]
      \item\label{condition1'} The gradient of $p$ vanishes at $0$, i.e., $\nabla p(0) = 0$, or equivalently
        \begin{equation*}
          \frac{\partial p(0)}{\partial t_{1}} = \cdots = \frac{\partial p(0)}{\partial t_{m}} = \frac{\partial p(0)}{\partial t_{m+1}} =0.
        \end{equation*}
      \item\label{condition2'}  For the Hessian matrix $A$ of the first $m$ variables,
        \begin{equation}
          \Det{A} = \Det{ \frac{\partial^{2} p(0)}{\partial t_{k} \partial t_{j}}}_{1 \leq k, j \leq m} \neq 0,
        \end{equation}
               and the other second-order partial derivatives read off
        \begin{equation*}
          \beta = \paren{\frac{\partial^{2} p(0)}{\partial t_{1} \partial t_{m+1}}, \cdots, \frac{\partial^{2} p(0)}{\partial t_{m} \partial t_{m+1}}}, \qquad  \alpha = \frac{\partial^{2} p(0)}{\partial t_{m+1}^{2}}.
        \end{equation*}
       \item\label{condition2.5'} \label{condition2.5} There exists $\lambda_0>0$ such that for every $\lambda\geq \lambda_0$, $I(\lambda)$ converges absolutely  throughout its range except for one neighbourhood of $0$.

      \item\label{condition3'}For every given $\varepsilon>0$ let  $D_{\eps} = \set{t \in D| \max\{\abs{t_1}, \ldots,\abs{t_{m+1}}\} < \varepsilon},$ there exists $\lambda_1>0$ such that
        $\varrho(\varepsilon,\lambda_1)>0$ where  \begin{multline}\varrho(\varepsilon,\lambda_1) =\\
         \inf\left\{\Re{p(t)-p(0)+\frac{f(t)-f(0)}{\sqrt{\lambda}}}: t\in D\backslash D_{\varepsilon}, \lambda\geq \lambda_1\right\}.\end{multline}
      \item\label{condition4'} $q(0) \neq 0$.
    \end{enumerate}
    Then
    \begin{align}
        I(\lambda)   \sim   (2\pi/\lambda)^{m/2} q(0) & e^{-\lambda p(0)-\sqrt{\lambda} f(0)} e^{\frac{1}{2} \widetilde{\eta} A^{-1} \widetilde{\eta}^{T}} \nonumber\\
        &\times \,\frac{i\pi \epsilon(A)}{\sqrt{\det{A}}}\mathrm{erf}\big( \frac{
        i\eta_{m+1} - i\beta A^{-1} \widetilde{\eta}^{T}
        }{ \sqrt{2\alpha - 2\beta A^{-1} \beta^{T}}
        }\big)
    \end{align}
    where $\mathrm{erf}(z)$ denotes the error function, and $\epsilon(A)$ equals to $1$ or $-1$ depending on the choice of the sign of $\sqrt{\Det{A}}$.
  \end{thm}

  \begin{proof}
    As in the proof of Theorem \ref{laplaceMethod1}, the integral can be split   into
    \begin{equation}
      I(\lambda) = \mathrm{P.V.} \int_{D_{\varepsilon}} (\cdot)+\mathrm{P.V.}  \int_{D \backslash D_{\varepsilon}} (\cdot):= I_{1}(\lambda) + I_{2}(\lambda).
    \end{equation}
    By Condition \ref{condition3'}, for sufficiently small $\varepsilon > 0$ and sufficiently large $\lambda$,  there exists a positive number $c$ independent of $\lambda$ such that $$\Re{p(t)-p(0)+\frac{f(t)-f(0)}{\sqrt{\lambda}}}\geq c$$
    for any $t \in D \backslash D_{\varepsilon}$. Therefore, by Condition \ref{condition2.5}
    \begin{align}\label{laplaceMethod2Decay}
      \abs{I_{2}(\lambda)} \leq K e^{-(\lambda-\lambda_0)(\Re{p(0)} + c)-(\sqrt{\lambda}-\sqrt{\lambda_0})\,\Re{f(0)}}
    \end{align}
    for some constant $K > 0$.

    Next, we turn to  the integral $I_{1}(\lambda)$. Write \begin{equation*}
          B = \begin{pmatrix}
            A & \beta^{T} \\
            \beta & \alpha
          \end{pmatrix},
        \end{equation*}
    then we know from Conditions \ref{condition1'}, \ref{condition2'} and \ref{condition3'} that $\Re{B}>0$. Furthermore, the Taylor expansion of  $p(t)$ at $t=0$ reads off $$p(t)=p(0)+\frac{1}{2} t \big(B + R(t)\big) t^{T}$$
    where $R(t)$ is a symmetric matrix whose entries are power series of $t$ with $R(0)=0$.

   As in the proof of Theorem \ref{laplaceMethod1}, rescaling the variables $t$ by $\frac{1}{\sqrt{\lambda}}$ we have the leading  asymptotic form
    \begin{align}
              I&_{1}(\lambda)\sim q(0) e^{-\lambda p(0)-\sqrt{\lambda} f(0)}\lambda^{-m/2}\, \mathrm{P.V.} \int_{\Real{R}^{m+1}} e^{-\frac{1}{2}   t B t^{T}} e^{-  t \eta^{T}} \frac{1}{t_{m+1}} d^{m+1}t\nonumber\\
              &=  q(0)   e^{-\lambda p(0)-\sqrt{\lambda} f(0)} e^{\frac{1}{2} \widetilde{\eta} A^{-1} \widetilde{\eta}^{T}}
      \Big(\frac{2\pi}{\lambda}\Big)^{\frac{m}{2}}\frac{i\pi}{\sqrt{\det{A}}}\mathrm{erf}\big( \frac{
        i\eta_{m+1} - i\beta A^{-1} \widetilde{\eta}^{T}
        }{ \sqrt{2\alpha - 2\beta A^{-1} \beta^{T}}
        }\big). \label{laplaceMethod2Domain}
        \end{align}
        In the last equality Lemma \ref{integallemma} below has been used.

    Combination of   \eqref{laplaceMethod2Decay} and \eqref{laplaceMethod2Domain} completes the proof of Theorem \ref{steepestDecentonEdge}.
  \end{proof}

  \begin{lem}\label{integallemma}
   Let $A$ be an $m \times m$  complex symmetric matrix,  $\beta$ be an  $m$-dimensional complex vector and $\alpha \in \Real{C}$. Write $t=(t_1,\ldots,t_{m+1})$ and $\eta: = (\widetilde{\eta},\eta_{m+1})\in \Real{C}^{m+1}$ where $\widetilde{\eta}$ is an  $m$-dimensional vector. If the real part of the block matrix \begin{equation*}
          B = \begin{pmatrix}
            A & \beta^{T} \\
            \beta & \alpha
          \end{pmatrix}
        \end{equation*}
    is positive definite, then     \begin{multline}
        \mathrm{P.V.} \int_{\Real{R}^{m+1}} e^{-\frac{1}{2} t B t^{T}} e^{-t \eta^{T}} \frac{1}{t_{m+1}} d^{m+1}t \\
        =(2\pi)^{m/2}  e^{\frac{1}{2} \widetilde{\eta} A^{-1} \widetilde{\eta}^{T}}
       \,\frac{i\pi}{\sqrt{\det{A}}}\mathrm{erf}\big(\frac{
        i\eta_{m+1} - i\beta A^{-1} \widetilde{\eta}^{T}
        }{ \sqrt{2\alpha - 2\beta A^{-1} \beta^{T}}
        }\big).
    \end{multline}
  \end{lem}
  \begin{proof}First, we get $\Re{A}>0$ from  $\Re{B}>0$. Since    $\Im\{A\}$ is symmetric by the assumption, we further have $\det A\neq 0$.
    Noticing the matrix decomposition
      \begin{equation*}
          B = \begin{pmatrix}
            I & 0\\
            \beta A^{-1} & 1
          \end{pmatrix} \begin{pmatrix}
            A & 0\\
            0 & \alpha - \beta A^{-1} \beta^{T}
          \end{pmatrix} \begin{pmatrix}
            I & A^{-1} \beta^{T}\\
            0 & 1
          \end{pmatrix},
      \end{equation*}
      we have
      \begin{equation*}
        \begin{split}
           t B t^{T} + 2t \eta^{T} &= (\widetilde{t} + t_{m+1} \beta A^{-1} + \widetilde{\eta} A^{-1}) A (\widetilde{t} + t_{m+1} \beta A^{-1} + \widetilde{\eta} A^{-1})^{T}+ \\
          & (\alpha - \beta A^{-1} \beta^{T}) t_{m+1}^{2} + 2t_{m+1} (\eta_{m+1} - \beta A^{-1} \widetilde{\eta}^{T}) - \widetilde{\eta} A^{-1} \widetilde{\eta}^{T}.
        \end{split}
      \end{equation*}
      Therefore, the Cauchy principal value integral
      \begin{align}
                 &\mathrm{P.V.} \int_{\Real{R}^{m+1}} e^{-\frac{1}{2} t B t^{T}} e^{-t \eta^{T}} \frac{1}{t_{m+1}} d^{m+1}t = (2\pi)^{m/2}  e^{\frac{1}{2} \widetilde{\eta} A^{-1} \widetilde{\eta}^{T}} \frac{1}{\sqrt{\Det{A}}}   \nonumber \\
          & \qquad\times \mathrm{P.V.} \int_{\Real{R}} e^{-\frac{1}{2} (\alpha - \beta A^{-1} \beta^{T}) t_{m+1}^{2}} e^{- (\eta_{m+1} - \beta A^{-1} \widetilde{\eta}^{T}) t_{m+1}} \frac{1}{t_{m+1}} dt_{m+1} \\
          &= (2\pi)^{m/2}  e^{\frac{1}{2} \widetilde{\eta} A^{-1} \widetilde{\eta}^{T}} \frac{1}{\sqrt{\Det{A}}}  \int_{0}^{\infty}dt_{m+1} e^{-\frac{1}{2} (\alpha - \beta A^{-1} \beta^{T}) t_{m+1}^{2}} \nonumber\\
          & \qquad \times \bigg(\frac{e^{- (\eta_{m+1} - \beta A^{-1} \widetilde{\eta}^{T}) t_{m+1}}}{t_{m+1}} - \frac{e^{(\eta_{m+1} - \beta A^{-1} \widetilde{\eta}^{T}) t_{m+1}}}{t_{m+1}}\bigg) \label{int1}\\
          &=(2\pi)^{m/2}  e^{\frac{1}{2} \widetilde{\eta} A^{-1} \widetilde{\eta}^{T}}
       \,\frac{i\pi}{\sqrt{\det{A}}}\mathrm{erf}\big( \frac{
        i\eta_{m+1} - i\beta A^{-1} \widetilde{\eta}^{T}
        }{ \sqrt{2\alpha - 2\beta A^{-1} \beta^{T}}
        }\big). \label{int2}
            \end{align}
      The last equality follows from  the series expression  of both \eqref{int1} and \eqref{int2} after integrating term-by-term; see also the equation 3.321.1 \cite[p.336]{jeffrey2007table}.
  \end{proof}

\begin{rem} \label{remark2.4} In order to use easily and widely, it is necessary to suppose that in Theorem \ref{laplaceMethod1} $f(\theta;t)=f_{\lambda}(\theta;t)$  does depend on the parameter $\lambda$.  Also if  $$\lim_{\lambda\rightarrow \infty}\sqrt{\lambda}\, \big(f_{\lambda}( \theta/\sqrt{\lambda};t/\lambda)-f_{\lambda}(0;0)\big)=\eta \theta^{T},$$ then the proof  of Theorem \ref{laplaceMethod1} is also applicable which  implies the same conclusion. Likewise,    if
$$\lim_{\lambda\rightarrow \infty}\sqrt{\lambda}\, \big(f( t/\sqrt{\lambda})-f(0)\big)=t  \eta^{T},$$
then  Theorem \ref{steepestDecentonEdge} still holds true.
  \end{rem}

\section{Products of Ginibre matrices with fixed parameters}\label{ginibrefixedsection}
\subsection{Integral representations of   finite sums}\label{repfinitesum} There are two main parts in the kernel function $K_{N}(z, z')$ given by  \eqref{kernelFunctionofProductInducedGaussianWeight}. One is the weight function $w_{m}(z)$ and its integral representation  in   \eqref{productWeightFunctionSpecifiedbyGinibre}  will be used to do asymptotic analysis, see Lemma \ref{laplaceMathedinbulk1} below. The other is the truncated finite sum of hypergeometric series
  \begin{align} T_{N}(z, z')
       = \sum_{l=0}^{N-1} \prod_{k=1}^{m} \frac{\Gamma(a_{k} + 1)}{\Gamma(a_{k} + l + 1)} (z\overline{z'})^{l} \label{finitesum}. \end{align}
  In order to get its asymptotics we need the following integral representation.
       \begin{prop} \label{repsumprop} For   a fixed   nonzero complex number $u$, let   $\rho=\rho(u)\neq 0$. Introduce rescaling variables
  \begin{equation} \label{rescalingz}
    z = u + \frac{v}{\rho \sqrt{N}}, \qquad z'= u + \frac{v'}{\rho \sqrt{N}}
  \end{equation} where  $v$ and $v'$ lie in a compact set of $\Real{C}$.  Then for sufficiently large $N$ we have
  \begin{align}
      T_{N}(N^{m/2}z, N^{m/2}z')  &=   (2\pi i)^{-m}N^{-\sum_{k=1}^{m}a_{k}}(z\overline{z'})^{-a_m}\prod_{j=1}^{m}  \Gamma(a_{j} + 1)\nonumber \\ &\times \int_{\mathcal{C}_{1} \times \cdots \times \mathcal{C}_{m}}Q(t) e^{-\sqrt{N} {F}(t)} \frac{e^{-N{P}_{1}(t)} - e^{-N{P}_{2}(t)}}{t_1 \cdots t_{m}-1}\, d^mt \label{saddleform}
    \end{align}
    where
  \begin{equation}\label{primativeP1}
    {P}_{1}(t) = -t_{1} - \cdots - t_{m-1} -  \abs{u}^{2}t_{m},
  \end{equation}
  \begin{equation}\label{primativeP2}
    {P}_{2}(t) = -t_{1} - \cdots - t_{m-1} -  \abs{u}^{2}t_{m} + \ln(t_1 \cdots t_{m}),
  \end{equation}
  \begin{equation}\label{primativeF}
    {F}(t) = - (u\overline{v'} + v\overline{u})t_{m}/\rho,
  \end{equation}
  and
  \begin{equation}\label{primativeQ}
    Q(t) =  e^{\frac{v\overline{v'}}{\rho^{2}}  t_{m}} \prod_{k=1}^{m} t_{k}^{-a_{k}}.
  \end{equation}
  Here    $\mathcal{C}_{k}$    is a path first going from  $-\infty$ to $r_{k}e^{i(-\pi+\theta_0)}$  ($0<\theta_0<\pi/2$) along the line parallel to the x-axis,  then going anticlockwise   along  the circle with radius of $r_k$  to $r_k e^{i(\pi-\theta_0)}$ and returning to $-\infty$ along the line parallel to the x-axis. Choose $r_1=\cdots=r_{m-1}=\abs{u}^{2/m}$ and $r_m=\abs{u}^{-2(m-1)/m}$.\end{prop}

  \begin{proof} With the help of the integral of the  reciprocal   Gamma function \cite[p.245]{whittaker1927course}
  \begin{equation}\label{intrepgamma1}
    \frac{1}{\Gamma(a)} = \frac{1}{2\pi i} \int_{\mathcal{C}} t^{-a}e^{t} dt,
  \end{equation}
  where $\mathcal{C}$ is a path starting at $-\infty$, encircling the origin once in the positive direction and returning to $-\infty$, with \eqref{finitesum} in mind we have   \begin{align}\label{partofKernelFunctionofProductInducedGaussianWeight}
       T_{N}&(z, z')
            = \prod_{j=1}^{m} \Gamma(a_{j} + 1) \sum_{l=0}^{N-1} (z\overline{z'})^{l} \prod_{k=1}^{m} \left(\frac{1}{2\pi i} \int_{\mathcal{C}_{k}} t_{k}^{-(a_{k} + l + 1)}e^{t_{k}} dt_{k}\right)\nonumber \\
      &=  (2\pi i)^{-m}\prod_{j=1}^{m}  \Gamma(a_{j} + 1) \int_{\mathcal{C}_{1} \times \cdots \times \mathcal{C}_{m}} \frac{1- \paren{\frac{z\overline{z'}}{t_{1} \cdots t_{m}}}^{N}}{1-\frac{z\overline{z'}}{t_{1} \cdots t_{m}}} \prod_{k=1}^{m} t_{k}^{-(a_{k} + 1)} e^{t_{k}}\, d^m t.
      \end{align}

      Since
  \begin{equation}
    z\overline{z'} = \abs{u}^{2} + \frac{u\overline{v'} + v\overline{u}}{\rho \sqrt{N}} + \frac{v\overline{v'}}{\rho^{2} N}
  \end{equation}
  and further $\Re{z\overline{z'}}>0$ for sufficiently large $N$,  simple calculation shows
  \begin{align}
      T_{N}(N^{m/2}z, N^{m/2}z')  &=   (2\pi i)^{-m}N^{-\sum_{k=1}^{m}a_{k}}(z\overline{z'})^{-a_m}\prod_{j=1}^{m}  \Gamma(a_{j} + 1)\nonumber \\ &\times \int_{\mathcal{C}_{1} \times \cdots \times \mathcal{C}_{m}}Q(t) e^{-\sqrt{N} {F}(t)} \frac{e^{-N{P}_{1}(t)} - e^{-N{P}_{2}(t)}}{t_1 \cdots t_{m}-1}\, d^mt \label{saddleform}
    \end{align}
    where we have used  the change of variables $t_{1} \to N t_{1}, \ldots, t_{m-1} \to N t_{m-1}$ and $ t_{m} \to N z\overline{z'}t_{m}$. The paths   can be chosen as needed.\end{proof}

\subsection{Several lemmas}
    First, we compute asymptotic behaviour of the weight function.
    \begin{lem}\label{laplaceMathedinbulk1}
      For $u\neq 0$, let
      \begin{equation}
        z = u + \frac{v}{\rho \sqrt{N}},
      \end{equation}
      then
      \begin{align}
          w_{m}(N^{m/2}z)&\sim\frac{1}{\pi\sqrt{m}} N^{\sum_{k=1}^{m} a_{k}} \paren{\frac{2\pi}{N}}^{(m-1)/2}\prod_{k=1}^{m} \frac{1}{\Gamma(a_{k}+1)} \nonumber \\
          &\quad \times e^{-Nm\abs{u}^{2/m}-\sqrt{N} \frac{u\overline{v} + v\overline{u}}{\rho} \abs{u}^{-2(m-1)/m}} e^{-\frac{\abs{v}^{2}}{\rho^{2}}\abs{u}^{-2(m-1)/m}} \nonumber\\
          &\quad \times e^{\frac{m-1}{2m} \paren{\frac{u\overline{v} + v\overline{u}}{\rho}}^{2} \abs{u}^{-4(m-1)/m} \abs{u}^{-2/m}} \abs{u}^{-\frac{m-1}{m}+\frac{2}{m} \sum_{k=1}^{m} a_{k}}.
            \end{align}
            Moreover, it holds true uniformly for $v$ in any compact subset of $\mathbb{C}$.
    \end{lem}
    \begin{proof} Note  that by rescaling   variables we have
      \begin{equation}
        \begin{split}
          w_{m}(N^{m/2}z) &= \frac{1}{\pi}\abs{z}^{2a_{m}}  2^{m-1}N^{\sum_{k=1}^{m} a_{k}} \prod_{k=1}^{m} \frac{1}{\Gamma(a_{k}+1)}\\
          &\quad \times \int_{\Real{R}_{+}^{m-1}} e^{-\sqrt{N}f(r)}e^{-Np(r)} q(r) d^{m-1}r
        \end{split}
      \end{equation}
      where
      \begin{equation}\label{pforWeightFunction2}
        p(r) = p(r_{1}, \cdots, r_{m-1}) = \frac{\abs{u}^{2}}{r_{1}^{2} \cdots r_{m-1}^{2}} + \sum_{k=1}^{m-1} r_{k}^{2},
      \end{equation}
      \begin{equation}\label{qforWeightFunction2}
        q(r) = q(r_{1}, \cdots, r_{m-1}) = \prod_{k=1}^{m-1} r_{k}^{2(a_{k} - a_{m}) - 1} e^{-\frac{v\overline{v}}{\rho^{2}} \frac{1}{r_{1}^{2} \cdots r_{m-1}^{2}}},
      \end{equation}
      and
      \begin{equation}\label{fforWeightFunction2}
        f(r) = f(r_{1}, \cdots, r_{m-1}) = \frac{u\overline{v} + v\overline{u}}{\rho} \frac{1}{r_{1}^{2} \cdots r_{m-1}^{2}}.
      \end{equation}
      One easily verifies   that the point $s = (\abs{u}^{1/m}, \cdots, \abs{u}^{1/m})$ is the unique saddle point of $p(r)$. 
      Let
      \begin{equation*}
        A = \paren{\frac{\partial^{2} p(s)}{\partial r_{k} \partial r_{j}}}_{1 \leq k,j \leq m-1},
      \end{equation*}
      where
      \begin{equation*}
        \frac{\partial^{2} p(s)}{\partial r_{k} \partial r_{j}} = \begin{cases}
          4, & \text{if~} k \neq j,\\
          8, & \text{if~} k = j,
        \end{cases}
      \end{equation*}
      and
      \begin{equation*}
        \eta := \nabla f(s) = -2 \frac{u\overline{v} + v\overline{u}}{\rho} \abs{u}^{-2(m-1)/m} \abs{u}^{-1/m}\paren{1, \cdots, 1}.
      \end{equation*}

     We see that $p(r) + \frac{1}{\sqrt{N}}f(r)$ has a positive lower bound  on the  region of integration except for a neighbour of $s$ for sufficiently  large $N$. As in the proof of Theorem \eqref{laplaceMethod1} (see Condition \ref{condition3}  therein), the integral except for the part in that neighbour of $s$ is exponentially decaying. Furthermore,
 Laplace's approximation \cite[p.495]{wong2001asymptotic} can be used to give
      \begin{align}
               w_{m}(N^{m/2}z)&\sim \frac{\abs{u}^{2a_{m}}}{\pi\sqrt{m}} N^{\sum_{k=1}^{m} a_{k}} \prod_{k=1}^{m} \frac{1}{\Gamma(a_{k}+1)} \paren{\frac{2\pi}{N}}^{(m-1)/2} \nonumber\\
          &\quad \times e^{-Np(s)} e^{-\sqrt{N}f(s)} e^{\frac{1}{2} \eta A^{-1} \eta^{T}} q(s)
              \end{align}
              from which the desired result follows. Here please refer to Appendix  \ref{appendixa} for the inverse $A^{-1}$.

                Obviously,   the asymptotics holds uniformly for $v$ in any compact subset of $\mathbb{C}$.
    \end{proof}

  We remark that the asymptotics of the weight function has been computed in \cite[Appendix B]{akemann2012universal} when $v=0$. Our computation is very similar to that. Actually, this special case has been  studied by Barnes \cite{barnes1907asymptotic},  and the readers can also find it in the Fields' work \cite{fields1972asymptotic}.

   \begin{lem} \label{twopartslemma} With the same notation as in Proposition \ref{repsumprop},   for sufficiently large $N$ there exists $c_0>0$ such that
  \begin{align}
      &T_{N}(N^{m/2}z, N^{m/2}z')  =   (2\pi i)^{-m}N^{-\sum_{k=1}^{m}a_{k}}(z\overline{z'})^{-a_m}\prod_{j=1}^{m}  \Gamma(a_{j} + 1)\nonumber \\ &  \times \Big(e^{Nm\abs{u}^{2/m}-c_0 N}  e^{\sqrt{N}\frac{u\overline{v'} + v\overline{u}}{\rho} \abs{u}^{-2(m-1)/m}}O(1) \nonumber\\
      &\hspace{3cm}+\int_{\gamma_{1} \times \cdots \times \gamma_{m}}q(t) e^{-\sqrt{N} f(t)} \frac{e^{-Np_{1}(t)} - e^{-Np_{2}(t)}}{t_{m}-1}\, d^mt\Big) \label{decayandleading}
    \end{align}
    where
  \begin{equation}\label{primativep1}
    p_{1}(t) = -t_{1} - \cdots - t_{m-1} -  \frac{\abs{u}^{2} t_{m}}{t_1 \cdots t_{m-1}},
  \end{equation}
  \begin{equation}\label{primativep2}
    p_{2}(t) = -t_{1} - \cdots - t_{m-1} -  \frac{\abs{u}^{2} t_{m}}{t_1 \cdots t_{m-1}} + \ln t_{m},
  \end{equation}
  \begin{equation}\label{primativef}
    f(t) = -  \frac{u\overline{v'} + v\overline{u}}{\rho} \frac{t_{m}}{t_1 \cdots t_{m-1}},
  \end{equation}
  and
  \begin{equation}\label{primativeq}
    q(t) =  e^{\frac{v\overline{v'}}{\rho^{2}} \, \frac{t_{m}}{t_1 \cdots t_{m-1}}} t_{m}^{-a_m}\prod_{k=1}^{m-1} t_{k}^{a_m-a_{k}-1}.
  \end{equation}
 Here    $\gamma_{k}=\{t_k:\abs{t_k}=\abs{u}^{2/m}, \abs{t_k-\abs{u}^{2/m}}<\varepsilon_{0}\}$ ($k=1,\ldots,m-1$) and
 $\gamma_{m}=\{t_m:\abs{t_m}=1, \abs{t_m-1}<\varepsilon_{0}\}$ for some $\varepsilon_{0}>0$. Moreover, both $\Re{p_{1}(t)}$ and $\Re{p_{2}(t)}$
 attain their unique minimum   over $\gamma_{1} \times \cdots \times \gamma_{m}$ at $\tilde{s} = (\abs{u}^{2/m}, \cdots, \abs{u}^{2/m}, 1)$.
 \end{lem}

   \begin{proof}  Let $D=\mathcal{C}_{1} \times \cdots \times \mathcal{C}_{m}$ and  $t_{0} = (\abs{u}^{2/m}, \cdots, \abs{u}^{2/m}, \abs{u}^{-2+2/m})$. Note that  for all $t\in D$ and $t \neq t_{0}$ the  following relations hold
      \begin{equation}
        \Re{P_{1}(t)} > \Re{P_{1}(t_{0})}\  \mathrm{and}\   \Re{P_{2}(t)} > \Re{P_{2}(t_{0})} \ \mathrm{for}\ t \neq t_{0}.
      \end{equation}
     Moreover, by choosing a small $\varepsilon_{0}>0$, setting $$D_{\varepsilon_{0}}=\Big\{t\in D :\abs{t_1\cdots t_m-1}<\varepsilon_0; \abs{t_k-\abs{u}^{2/m}}<\varepsilon_{0}, 1\leq k< m-1\Big\},$$ one can verify  that  there exists a positive number $c_0$ such that for sufficiently large $N$ and for every $t\in D \backslash D_{\varepsilon_0}$
     $$ \Re{P_1(t)-P_1(t_0)+\frac{F(t)-F(t_0)}{\sqrt{N}}} \geq 2c_0$$
     and   $$ \Re{P_2(t)-P_2(t_0)+\frac{F(t)-F(t_0)}{\sqrt{N}}} \geq 2c_0.$$
     We also know from the fact $\abs{ t_1 \cdots t_{m}}\geq 1$ for all $t\in D$  that \begin{align}
      \abs{\frac{1- e^{-N (P_{2}(t)-{P}_{1}(t) )}}{t_1 \cdots t_{m}-1}} &= \abs{\frac{1- (t_1 \cdots t_{m})^{-N}}{t_1 \cdots t_{m}-1}}  \leq N,
    \end{align}
    which can be controlled by $e^{-c_0 N}$.
     These  show that the integral over $D \backslash D_{\varepsilon_0}$ is exponentially decaying  as indicated in \eqref{decayandleading}.

    We next turn to the remaining integral.   Noticing the choice of the contours $\mathcal{C}_{k}$ and the definition of $D_{\varepsilon_{0}}$, we  can always choose a small $\varepsilon_0$ such that the range of $t_1\cdots t_m$ is an arc of the unit circle. By  change of variables $$(t_{1}, \cdots, t_{m-1}, t_{m}) \mapsto (t_{1}, \cdots, t_{m-1}, \frac{t_{m}}{t_{1} \cdots t_{m-1}}),$$ the domain $D_{\varepsilon_{0}}$ is changed into $\gamma_{1} \times \cdots \times \gamma_{m}$, and the extremum property of $\Re{p_{1}(t)}$ and $\Re{p_{2}(t)}$ follows from that of $\Re{P_{1}(t)}$ and $\Re{P_{2}(t)}$ over $D_{\varepsilon_0}$.
    Then the lemma immediately follows.
      \end{proof}

    \begin{lem}\label{laplaceMathedinbulk2}
     For $u\neq 0$, with $z, z'$ in \eqref{rescalingz} we have
      \begin{align}
               &(z\overline{z'})^{-a_m}\,\mathrm{P.V.} \int_{\gamma_{1} \times \cdots \times \gamma_{m}}q(t) e^{-\sqrt{N} {f}(t)} \frac{e^{-N{p}_{1}(t)}}{t_{m}-1}\, d^mt\nonumber\\
          &\sim i^{m} \frac{\pi}{\sqrt{m}}     (2\pi/N)^{(m-1)/2} \abs{u}^{-\frac{m-1}{m}-\frac{2}{m}\sum_{k=1}^{m} a_{k}} e^{Nm\abs{u}^{2/m}+\sqrt{N}\frac{u\overline{v'} + v\overline{u}}{\rho} \abs{u}^{-2(m-1)/m}}\nonumber\\
           &  \quad \times e^{\frac{v\overline{v'}}{\rho^{2}} \abs{u}^{-2(m-1)/m}} e^{-\frac{m-1}{2m} \paren{\frac{u\overline{v'} + v\overline{u}}{\rho}}^{2} \abs{u}^{-4(m-1)/m} \abs{u}^{-2/m}} .
            \end{align}
      where $p_{1}(t)$, $f(t)$ and $q(t)$ are given by (\ref{primativep1}), (\ref{primativef}) and (\ref{primativeq}) respectively.
       Moreover, it holds uniformly for $v, v'$ in any compact subset of $\mathbb{C}$.
    \end{lem}
    \begin{proof}
     Note that   $\tilde{s}: = (\abs{u}^{2/m}, \cdots, \abs{u}^{2/m}, 1)$ is the unique point lying in $\gamma_{1} \times \cdots \times \gamma_{m}$ such that
      \begin{equation}
        \Re{p_{1}(t)} > \Re{p_{1}(\tilde{s})},\quad t \neq \tilde{s}.
      \end{equation}
     Moreover,  for $p_{1}(t)$  Condition \ref{condition3} in Theorem \ref{laplaceMethod1} holds from Lemma \ref{twopartslemma} and others are easy to verify. By Theorem \ref{laplaceMethod1} we could  do some  computation and thus complete the proof.

     Obviously,       \begin{equation*}
        \frac{\partial p_{1}(\tilde{s})}{\partial t_{k}} = \begin{cases}
          0 ,& 1\leq k<m,\\
          -\abs{u}^{2/m} ,& k = m,
        \end{cases}
      \end{equation*}
      and the Hessian of the first $m-1$ variables reads off
      \begin{equation*}
        A = \paren{\frac{\partial^{2} p_{1}(\tilde{s})}{\partial t_{k} \partial t_{j}}}_{1 \leq k,j \leq m-1}
      \end{equation*}
      where
      \begin{equation*}
        \frac{\partial^{2} p_{1}(\tilde{s})}{\partial t_{k} \partial t_{j}} = \begin{cases}
          -\abs{u}^{-2/m} ,& k \neq j,\\
          -2\abs{u}^{-2/m} ,& k = j.
        \end{cases}
      \end{equation*}
      We also have
      \begin{equation*}
        \frac{\partial f(\tilde{s})}{\partial t_{k}} = \frac{u\overline{v'} + v\overline{u}}{\rho} \abs{u}^{-2}.
      \end{equation*}
      Write $\eta = (\frac{\partial f(\tilde{s})}{\partial t_{1}}, \cdots, \frac{\partial f(\tilde{s})}{\partial t_{m-1}})$, application of Theorem \ref{laplaceMethod1} gives
      \begin{equation*}
        \begin{split}
          & (z\overline{z'})^{-a_m}\,\mathrm{P.V.} \int_{\gamma_{1} \times \cdots \gamma_{m}} {q}(t) e^{-\sqrt{N} {f}(t)} e^{-N{p}_{1}(t)} \frac{1}{t_{m}-1} dt \\
          &\sim  (-1)^{m-1} \abs{u}^{-a_m}i\pi    (2\pi/N)^{(m-1)/2} q(\tilde{s}) e^{-N p_{1}(\tilde{s})-\sqrt{N} f(\tilde{s})} e^{\frac{1}{2} \eta A^{-1} \eta^{T}} \frac{1}{\sqrt{\det{A}}}
        \end{split}
      \end{equation*}
      from which the lemma follows.
    \end{proof}

    \begin{lem}\label{laplaceMathedinbulk3}Let
           \begin{equation}
        \Phi := (z\overline{z'})^{-a_m}\,\mathrm{P.V.} \int_{\gamma_{1} \times \cdots \times \gamma_{m}}q(t) e^{-\sqrt{N} {f}(t)} \frac{e^{-N{p}_{2}(t)}}{t_{m}-1}\, d^mt        \end{equation}
      where $p_{2}(t)$, $f(t)$ and $q(t)$ are given by (\ref{primativep2}), (\ref{primativef}) and (\ref{primativeq}) respectively.
      Then the followings hold   uniformly for $v, v'$ in any compact subset of $\mathbb{C}$.
      \begin{enumerate}[label = $\mathrm{(\Roman*)}$]
        \item\label{ginibresum1} For $\abs{u} \neq 0, 1$,
          \begin{align}
                        \Phi &\sim i^{m} \frac{\pi}{\sqrt{m}}    (2\pi/N)^{(m-1)/2} \abs{u}^{-\frac{m-1}{m}-\frac{2}{m}\sum_{k=1}^{m} a_{k}} e^{Nm\abs{u}^{2/m}+\sqrt{N}\frac{u\overline{v'} + v\overline{u}}{\rho} \abs{u}^{-2(m-1)/m}}\nonumber\\
              &\quad \times e^{\frac{v\overline{v'}}{\rho^{2}} \abs{u}^{-2(m-1)/m}}  e^{-\frac{m-1}{2m} \paren{\frac{u\overline{v'} + v\overline{u}}{\rho}}^{2} \abs{u}^{-4(m-1)/m} \abs{u}^{-2/m}}  \mathrm{sign}(\abs{u}^{2/m} - 1).
                     \end{align}
        \item\label{ginibresum2} For $\abs{u} = 1$,
          \begin{align}
              \Phi &\sim i^{m} \frac{\pi}{\sqrt{m}}    (2\pi/N)^{(m-1)/2} \abs{u}^{-\frac{m-1}{m}-\frac{2}{m}\sum_{k=1}^{m} a_{k}}   e^{Nm\abs{u}^{2/m}+\sqrt{N}\frac{u\overline{v'} + v\overline{u}}{\rho} \abs{u}^{-2(m-1)/m}}\nonumber\\
              &\quad \times e^{\frac{v\overline{v'}}{\rho^{2}} \abs{u}^{-2(m-1)/m}}   e^{-\frac{m-1}{2m} \paren{\frac{u\overline{v'} + v\overline{u}}{\rho}}^{2} \abs{u}^{-4(m-1)/m} \abs{u}^{-2/m}}  \mathrm{erf}(\frac{u\overline{v'} + v\overline{u}}{\rho\sqrt{2m}}).
                  \end{align}
      \end{enumerate}
    \end{lem}
    \begin{proof}
      It is easy to check that  $\tilde{s}: = (\abs{u}^{2/m}, \cdots, \abs{u}^{2/m}, 1)$ is the unique point such that
      \begin{equation}
        \Re{p_{2}(t)} > \Re{p_{2}(\tilde{s})},\quad t \neq \tilde{s}.
      \end{equation}
      And moreover,  for $p_{2}(t)$  Condition \ref{condition3'} in Theorem \ref{steepestDecentonEdge} holds from Lemma \ref{twopartslemma} and others are easy to verify. Our computation is as follows.

      \textbf{Case \ref{ginibresum1}}. Simple calculation shows
      \begin{equation*}
        \frac{\partial p_{2}(\tilde{s})}{\partial t_{k}} = \begin{cases}
          0 ,& 1\leq k <m,\\
          1-\abs{u}^{2/m} ,& k = m,
        \end{cases}
      \end{equation*}
     and the Hessian of first $m-1$ variables
      \begin{equation*}
        A = \paren{\frac{\partial^{2} p_{2}(\tilde{s})}{\partial t_{k} \partial t_{j}}}_{1 \leq k,j \leq m-1},
      \end{equation*}
      where
      \begin{equation*}
        \frac{\partial^{2} p_{2}(\tilde{s})}{\partial t_{k} \partial t_{j}} = \begin{cases}
          -\abs{u}^{-2/m} ,& k \neq j,\\
          -2\abs{u}^{-2/m} ,& k = j.
        \end{cases}
      \end{equation*}

     We also have
      \begin{equation*}
        \frac{\partial f(\tilde{s})}{\partial t_{k}} = \begin{cases}
          \frac{u\overline{v'} + v\overline{u}}{\rho} \abs{u}^{-2}, & 1\leq k<m,\\
          -\frac{u\overline{v'} + v\overline{u}}{\rho} \abs{u}^{-2(m-1)/m}, & k =m.
        \end{cases}
      \end{equation*}
      Let $\eta = (\frac{\partial f(\tilde{s})}{\partial t_{1}}, \cdots, \frac{\partial f(\tilde{s})}{\partial t_{m-1}})$ and $\eta' = \frac{\partial f(\tilde{s})}{\partial t_{m}}$.  For $\abs{u} \neq 1$,  application of  Theorem \ref{laplaceMethod1}   gives
      \begin{equation*}
        \begin{split}
          \Phi&\sim (-1)^{m-1} \abs{u}^{-2a_m} i\pi    (2\pi/N)^{(m-1)/2} q(\tilde{s}) e^{-N p_{2}(\tilde{s})-\sqrt{N} f(\tilde{s})} e^{\frac{1}{2} \eta A^{-1} \eta^{T}} \frac{\mathrm{sign}(\abs{u}^{2/m} - 1)}{\sqrt{\det{A}}}
        \end{split}
      \end{equation*}
      from which Statement \ref{ginibresum1}  follows.

\textbf{Case \ref{ginibresum2}}. Furthermore, we have
      \begin{equation*}
        \beta = (\frac{\partial^{2} p_{2}(\tilde{s})}{\partial t_{1} \partial t_{m}}, \cdots, \frac{\partial^{2} p_{2}(\tilde{s})}{\partial t_{m-1} \partial t_{m}}) = (1, \cdots, 1),
      \end{equation*}
      and
      \begin{equation*}
        \alpha = \frac{\partial^{2} p_{2}(\tilde{s})}{\partial t_{m}^{2}} = -1.
      \end{equation*}
      When $\abs{u}=1$,  by Theorem \ref{steepestDecentonEdge}  we get
      \begin{equation*}
        \begin{split}
          \Phi
          &\sim (-1)^{m-1} \abs{u}^{-2a_m} (2\pi/N)^{(m-1)/2}  q(\tilde{s}) e^{-N p(\tilde{s}) - \sqrt{N} f(\tilde{s})}  e^{\frac{1}{2} {\eta} A^{-1} {\eta}^{T}} \\
          &\quad \times  \,\frac{i\pi}{\sqrt{\det{A}}}\mathrm{erf}\big( \frac{
       i\eta' - i\beta A^{-1} {\eta}^{T}
        }{ \sqrt{2\alpha - 2\beta A^{-1} \beta^{T}}
        }\big)
        \end{split}
      \end{equation*}from which  Statement \ref{ginibresum2}  follows.
    \end{proof}
\subsection{Limiting eigenvalue density}
  In this subsection we prove pointwise convergence of one-point correlation functions as $N$ goes to infinity.  In the parameter-fixed case (cf. Theorem \ref{limitdensityfixed} below),   this result has been given in \cite{akemann2012universal}; see  e.g. \cite{burda2010spectrum} for the different convergence mode.
  \begin{thm}\label{limitdensityfixed}
    For  fixed $a_{1}, \ldots, a_{m} > -1$,    the limiting eigenvalue density
    \begin{equation}
      R_{1}(z): = \lim_{N \to \infty} N^{m-1}R_{N,1}(N^{m/2}z) = \frac{1}{m\pi} \abs{z}^{\frac{2}{m} -2} \chi_{\set{u:\abs{u} < 1}}(z)
    \end{equation}
    holds true for any complex $z\neq 0$  where    $\chi_{\set{u:\abs{u} < 1}}(z)$ is defined as in  \eqref{characteristicfuncion2}.
  \end{thm}

  \begin{proof}
    Set  $v=0$ in Lemma \ref{laplaceMathedinbulk1}, we obtain
    \begin{equation}
      \begin{split}
        w_{m}(N^{m/2}z) &\sim \frac{1}{\pi\sqrt{m}} ( 2\pi/N)^{(m-1)/2} N^{\sum_{k=1}^{m}a_{k}} \prod_{k=1}^{m} \frac{1}{\Gamma(a_{k}+1)} \\
        &\quad \times  e^{-Nm\abs{u}^{2/m}} \abs{u}^{-\frac{m-1}{m}+\frac{2}{m} \sum_{k=1}^{m} a_{k}}.\label{density1}
      \end{split}
    \end{equation}

    Set $v=v'=0$ in   Lemmas \ref{twopartslemma}, \ref{laplaceMathedinbulk2} and \ref{laplaceMathedinbulk3}, together they show
    \begin{equation}
      \begin{split}
        T_{N}(N^{m/2}z, N^{m/2}z)&\sim \frac{1}{\sqrt{m}  }N^{-\sum_{k=1}^{m}a_{k}} (2\pi N)^{-(m-1)/2}\prod_{k=1}^{m}  \Gamma(a_{k} + 1)   \\
        &\quad \times   \abs{u}^{-\frac{m-1}{m}-\frac{2}{m} \sum_{k=1}^{m} a_{k}} e^{Nm\abs{u}^{2/m}} \chi_{\set{z:\abs{z} < 1}}(u).\label{density2}
      \end{split}
    \end{equation}

    Since    $R_{N,1}(z)=w_{m}(z)T_{N}(z,z)$ by definition, combination of \eqref{density1} and \eqref{density2} completes the proof.
  \end{proof}

 \subsection{Proof of Theorem \ref{univeralitytheoremfixed}}

  \begin{proof}
    Let
    \begin{equation}
      \psi_{N}(v) = e^{-\sqrt{N} \frac{u\overline{v} - v\overline{u}}{\rho} \abs{u}^{-2(m-1)/m}} e^{\frac{m-1}{4m} \frac{(u\overline{v})^{2} - (v\overline{u})^{2}}{\rho^{2}} \abs{u}^{-4 + 2/m}},
    \end{equation}
    define  a diagonal  matrix  by  $D = \mathrm{diag}(\psi_{N}(v_{1}), \cdots, \psi_{N}(v_{n}))$.

    For $0<\abs{u}<1$,   combining Lemmas \ref{twopartslemma}, \ref{laplaceMathedinbulk2}  and Statement \ref{ginibresum1} in   Lemma \ref{laplaceMathedinbulk3}, we obtain
    \begin{multline}
      T_{N}(N^{m/2}z_{k},N^{m/2}z_{j}) \sim (2\pi)^{-(m-1)/2}N^{-(m-1)/2-\sum_{k=1}^{m}a_{k}} \prod_{j=1}^{m}  \Gamma(a_{j} + 1) \\
      \times \frac{1}{\sqrt{m}} \abs{u}^{-\frac{m-1}{m}-\frac{2}{m}\sum_{k=1}^{m} a_{k}} e^{\frac{v\overline{v'}}{\rho^{2}} \abs{u}^{-2(m-1)/m}} \\
      \times \mathrm{exp}\Big\{Nm\abs{u}^{2/m}+\sqrt{N}\frac{u\overline{v'} + v\overline{u}}{\rho} \abs{u}^{-2(m-1)/m}\Big\}\\
      \times \mathrm{exp}\Big\{-\frac{m-1}{2m} \big(\frac{u\overline{v'} + v\overline{u}}{\rho}\big)^{2} \abs{u}^{-4(m-1)/m} \abs{u}^{-2/m}\Big\}.
    \end{multline}

    Now, combining the asymptotics of $T_{N}$ and Lemmas \ref{laplaceMathedinbulk1}, with the correlation kernel  \eqref{kernelFunctionofProductInducedGaussianWeight} in mind we find
    \begin{multline}
           \psi_{N}(v_{k}) K_{N}(N^{m/2}z_{k},N^{m/2}z_{j}) \psi_{N}^{-1}(v_{j})
        \\ \sim N^{-(m-1)} \frac{1}{m\pi} \abs{u}^{-2(m-1)/m} e^{-\frac{1}{2} (\abs{v_{k}}^{2} + \abs{v_{j}}^{2} - 2v_{k}\overline{v_{j}})}.
        \end{multline}
    Thus,
    \begin{equation}
      \begin{split}
        R_{N,n}(N^{m/2}z) &= \Det{D(K_{N}(N^{m/2}z_{k}, N^{m/2}z_{j}))D^{-1}}\\
        &\sim N^{-n(m-1)} {\rho^{2n}} \Det{\frac{1}{\pi} e^{-\frac{1}{2} (\abs{v_{k}}^{2} + \abs{v_{j}}^{2} - 2v_{k}\overline{v_{j}})}}_{1 \leq k, j \leq n}.
      \end{split}
    \end{equation}
    The bulk limit  of \eqref{scalingLimitBulk} then follows.

    Likewise, for $\abs{u}=1$, by Lemmas \ref{twopartslemma}, \ref{laplaceMathedinbulk2}  and Statement \ref{ginibresum2} in   Lemma \ref{laplaceMathedinbulk3},   we know
    \begin{multline}
      T_{N}(N^{m/2}z_{k},N^{m/2}z_{j}) \sim (2\pi)^{-(m-1)/2}N^{-(m-1)/2-\sum_{k=1}^{m}a_{k}} \prod_{j=1}^{m}  \Gamma(a_{j} + 1) \\
      \times \frac{1}{2\sqrt{m}} \abs{u}^{-\frac{m-1}{m}-\frac{2}{m}\sum_{k=1}^{m} a_{k}} e^{\frac{v\overline{v'}}{\rho^{2}} \abs{u}^{-2(m-1)/m}} \mathrm{erf}(\frac{u\overline{v'} + v\overline{u}}{\rho\sqrt{2m}})\\
      \times \mathrm{exp}\Big\{Nm\abs{u}^{2/m}+\sqrt{N}\frac{u\overline{v'} + v\overline{u}}{\rho} \abs{u}^{-2(m-1)/m}\Big\}\\
      \times \mathrm{exp}\Big\{-\frac{m-1}{2m} \big(\frac{u\overline{v'} + v\overline{u}}{\rho}\big)^{2} \abs{u}^{-4(m-1)/m} \abs{u}^{-2/m}\Big\}.
    \end{multline}
    Thus, combining the Lemma \ref{laplaceMathedinbulk1}, we get
    \begin{equation}
      \begin{split}
        &\psi_{N}(v_{k}) K_{N}(N^{m/2}z_{k},N^{m/2}z_{j}) \psi_{N}^{-1}(v_{j}) \\
        &\sim N^{-(m-1)} \frac{1}{2\pi} \rho^{2} e^{-\frac{1}{2} (\abs{v_{k}}^{2} + \abs{v_{j}}^{2} - 2v_{k}\overline{v_{j}})} \mathrm{erfc}(\frac{u\overline{v_{j}} + v_{k} \overline{u}}{\sqrt{2}})
      \end{split}
    \end{equation}
   from which the edge limit follows.

   Uniform convergence   on any compact subset  of $\mathbb{C}$    follows from that of the results in the involved lemmas.
  \end{proof}

  \section{Products of Ginibre matrices with varying parameters}\label{varyingginibre}
    In this  section  we turn to the product of Ginibre matrices with varying parameters  $a_{k} = \delta_{k} N, k=1,\ldots,m$. Without loss of generality, we assume that  $$\delta_{m} = \min \set{\delta_{k}: k = 1, \cdots, m} \geq 0.$$

    As in the fixed parameters case we first prove a series of lemmas and then complete the proofs of  main results.  For the latter use, let $\xi(u)$ be the unique positive solution of algebraic equation
    \begin{equation}\label{polynomialEquation}
      x \prod_{k=1}^{m-1} \paren{(\delta_{k} - \delta_{m}) + \abs{u}^{2} x} - 1 = 0
    \end{equation}
    for any nonzero $u$ in the complex plane.
\subsection{Several lemmas}
    \begin{lem}\label{vWeight} Suppose that $a_{k} = \delta_{k} N, k=1,\ldots,m$. Let\begin{equation}
        s_{k} = \sqrt{\delta_{k} - \delta_{m} + \abs{u}^{2} \xi(u)}, \quad k=1, \cdots, m-1,\label{s-constants}
      \end{equation}
   and  introduce a rescaling variable
      \begin{equation}
        z = u + \frac{v}{\rho \sqrt{N}}
      \end{equation}
      where $\rho$ is a real nonzero  parameter.
      Then  as $N\rightarrow \infty$ we have for $u\neq 0$
      \begin{align}
              w_{m}(N^{m/2} z) &\sim  \frac{\sqrt{\xi(u)}}{\pi \sqrt{1 + \abs{u}^{2} \xi(u) \sum_{k=1}^{m-1} {s_{k}^{-2}}}} \paren{\frac{2\pi}{N}}^{(m-1)/2}\prod_{k=1}^{m} \frac{N^{a_{k}}}{\Gamma(a_{k}   + 1)} \nonumber\\
              &\quad \times  \exp\Big\{-N m\abs{u}^{2}\xi(u)   -N \sum_{k=1}^{m-1} (\delta_{k} - \delta_{m}) \paren{1 - 2\ln s_{k}}\Big\}\nonumber\\
              &\quad \times \exp\Big\{\frac{1}{2} \xi^{2}(u) \frac{(u\overline{v} + v\overline{u})^{2}}{\rho^{2}} \frac{1}{1 + \sum_{k=1}^{m-1}  \abs{u}^{2} \xi(u) s_{k}^{-2} } \sum_{k=1}^{m-1}  s_{k}^{-2}\Big\}\nonumber\\
              &\quad \times
              \abs{z}^{2a_{m}} \exp\Big\{- \frac{\abs{v}^{2}}{\rho^{2}} \xi(u) -\sqrt{N} \frac{u\overline{v} + v\overline{u}}{\rho} \xi(u)\Big\}.
             \end{align}
 Moreover, it holds true uniformly for $v$ in any compact subset of $\mathbb{C}$.
    \end{lem}
    \begin{proof}
      The weight function (\ref{productWeightFunctionSpecifiedbyGinibre}) can be rewritten as
      \begin{multline}
                 w_{m}(N^{m/2} z)  =  \abs{z}^{2a_{m}} \frac{2^{m-1}}{\pi} \prod_{k=1}^{m} \frac{N^{ a_{k}  }}{\Gamma(a_{k}  + 1)} \\
           \quad \times \int_{\Real{R}_{+}^{m-1}} q(r) e^{-\sqrt{N} f(r)} e^{-N p(r)} d^{m-1}r,
             \end{multline}
      where
      \begin{equation*}
        p(r) = \sum_{k=1}^{m-1} \big(r_{k}^{2} - 2(\delta_{k} - \delta_{m}) \ln{r_{k}}\big) + \frac{\abs{u}^{2}}{\prod_{k=1}^{m-1} r_{k}^{2}},
      \end{equation*}
      \begin{equation*}
        q(r) = \exp\big\{- \frac{\abs{v}^{2}}{\rho^{2}} \frac{1}{\prod_{k=1}^{m-1} r_{k}^{2}}\big\} \,\prod_{k=1}^{m-1} \frac{1}{r_{k}},
      \end{equation*}
      and
      \begin{equation*}
        f(r) = \frac{u\overline{v} + v\overline{u}}{\rho}  \prod_{k=1}^{m-1}  \frac{1}{r_{k}^{2}}.
      \end{equation*}

     We know from the first derivatives of $p(r)$
      \begin{equation*}
        \frac{\partial p}{\partial r_{k}} = 2r_{k} - 2(\delta_{k} - \delta_{m}) \frac{1}{r_{k}} - 2\frac{\abs{u}^{2}}{r_{k} \prod_{j=1}^{m-1} r_{j}^{2}}
      \end{equation*}
that $s := (s_{1}, \cdots, s_{m-1})$
      is the unique saddle point of $p(r)$ in $\Real{R}_{+}^{m-1}$. Also noting \eqref{polynomialEquation} we have
      \begin{equation*}
        p(s) = m\abs{u}^{2}\xi(u) + \sum_{k=1}^{m-1} (\delta_{k} - \delta_{m}) \paren{1 - \ln\big(\delta_{k} - \delta_{m}  + \abs{u}^{2} \xi(u)\big)}.
      \end{equation*}

      On the other hand, we can easily calculate the Hessian matrix of $p(r)$  at the point $s=(s_{1}, \cdots, s_{m-1})$ as follows
      \begin{equation*}
        A := \paren{\frac{\partial^{2} p(s)}{\partial r_{k} \partial r_{j}}}_{1 \leq k, j \leq m-1}
      \end{equation*}
      where
      \begin{equation*}
        \frac{\partial^{2} p(s)}{\partial r_{k} \partial r_{j}} = \begin{cases}
          4\abs{u}^{2} \xi(u) \frac{1}{s_{k} s_{j}}, & k \neq j;\\
          4 + 4 \abs{u}^{2} \xi(u) \frac{1}{s_{k}^{2}}, & k = j.
        \end{cases}
      \end{equation*}
          Therefore, $A$ is positive definite, its inverse reads off (cf. Appendix \ref{appendixa})
      \begin{align}
       & A^{-1} = \frac{1}{4} \frac{1}{1 + \abs{u}^{2} \xi(u) \sum_{k=1}^{m-1} {s_{k}^{-2}}}\times \nonumber \\
        & \begin{pmatrix}
          1 + \sum_{k \neq 1} \frac{\abs{u}^{2} \xi(u)}{s_{k}^{2}} &- \frac{\abs{u}^{2} \xi(u)}{s_{1}s_{2}}& \cdots & - \frac{\abs{u}^{2} \xi(u)}{s_{1}s_{m-1}}\\
         - \frac{\abs{u}^{2} \xi(u)}{s_{2}s_{1}}   &1 + \sum_{k \neq 2} \frac{\abs{u}^{2} \xi(u)}{s_{k}^{2}}& \cdots & - \frac{\abs{u}^{2} \xi(u)}{s_{2}s_{m-1}}\\
         \vdots& \vdots & \ddots & \vdots \\
          - \frac{\abs{u}^{2} \xi(u)}{s_{m-1}s_{1}} &- \frac{\abs{u}^{2} \xi(u)}{s_{m-1}s_{1}}& \cdots & 1 + \sum_{k \neq m-1} \frac{\abs{u}^{2} \xi(u)}{s_{k}^{2}}
        \end{pmatrix}
      \end{align}
      and the determinant of the Hessian   at $s$ is
      \begin{equation*}
        \begin{split}
          \Det{A}  = 4^{m-1} \Big(1 + \abs{u}^{2} \xi(u) \sum_{k=1}^{m-1} {s_{k}^{-2}}\Big).
        \end{split}
      \end{equation*}

      Similarly, we can get  $$
        f(s) = \xi(u) \frac{u\overline{v} + v\overline{u}}{\rho}$$
      and  the gradient of $f(r)$ at $s$
      \begin{equation*}
        \eta := \nabla f(s)=-2 \xi(u) \frac{u\overline{v} + v\overline{u}}{\rho} \Big(\frac{1}{s_{1}}, \cdots, \frac{1}{s_{m-1}}\Big).
      \end{equation*}
      Thus simple manipulation shows
      \begin{equation*}
        \eta A^{-1} \eta^{T} = \xi^{2}(u) \frac{(u\overline{v} + v\overline{u})^{2}}{\rho^{2}} \frac{1}{1 + \abs{u}^{2} \xi(u) \sum_{k=1}^{m-1} s_{k}^{-2}} \sum_{k=1}^{m-1} {s_{k}^{-2}}.
      \end{equation*}

      We see that $p(r) + \frac{1}{\sqrt{N}}f(r)$ has a positive lower bound  on the  region of integration except for a neighbour of $s$ for sufficiently  large $N$. As in the proof of Theorem \eqref{laplaceMethod1} (see Condition \ref{condition3}  therein), the integral except for the part in that neighbour of $s$ is exponentially decaying. Furthermore, Laplace's approximation \cite[p.495]{wong2001asymptotic} can be used to give
      \begin{equation*}
        \begin{split}
          w_{m}(N^{m/2} z) &\sim \abs{z}^{2a_{m}} \frac{2^{m-1}}{\pi} \prod_{k=1}^{m} \frac{N^{ a_{k}  }}{\Gamma(a_{k}  + 1)} \\
          &\times \,\paren{\frac{2\pi}{N}}^{(m-1)/2}  \big(\det{A}\big)^{-1/2} q(s)       e^{-Np(s) -\sqrt{N}f(s)} e^{\frac{1}{2}  \eta A^{-1} \eta^{T}}
        \end{split}
      \end{equation*}
       from which the desired result follows.

                Obviously, the asymptotics holds true uniformly for $v$ in any compact subset of $\mathbb{C}$.
    \end{proof}

    Next we turn to consider the asymptotics of the truncated series given in \eqref{finitesum}.
    \begin{lem} \label{twopartslemmavarying} For a fixed nonzero complex number $u$, let $\rho=\rho(u)\neq 0$. Suppose that $a_{j} = \delta_{j} N$ where $\delta_{j} \geq 0$ ($j=1,\ldots,m$) and
    \begin{equation}
      \tilde{s}_{k} =  \delta_{k} - \delta_{m} + \abs{u}^{2} \xi(u), \quad k=1, \cdots, m-1\  \text{while}\   \tilde{s}_{m}=1.\label{tildes-constants}
    \end{equation}
         Introduce rescaling variables
  \begin{equation} \label{rescalingzvarying}
    z = u + \frac{v}{\rho \sqrt{N}}, \qquad z'= u + \frac{v'}{\rho \sqrt{N}}
  \end{equation} where  $v,v'$ lie in a compact set of $\Real{C}$,   then for sufficiently large $N$ there exists $c_0>0$ such that
  \begin{align}
    &T_{N}(N^{m/2}z, N^{m/2}z')  =   (2\pi i)^{-m}N^{-\sum_{k=1}^{m}a_{k}}(z\overline{z'})^{-a_m}\prod_{j=1}^{m}  \Gamma(a_{j} + 1)\nonumber \\ &  \times \Big(e^{-Np_{1}(\tilde{s})-c_0 N}  e^{-\sqrt{N} f(\tilde{s})}O(1) \nonumber\\
    &\hspace{3cm}+\int_{\gamma_{1} \times \cdots \times \gamma_{m}}q(t) e^{-\sqrt{N} f(t)} \frac{e^{-Np_{1}(t)} - e^{-Np_{2}(t)}}{t_{m}-1}\, d^mt\Big) \label{decayandleadingvarying}
  \end{align}
  where
  \begin{equation}\label{primativep1varying}
  {p}_{1}(t) = -\sum_{k=1}^{m-1}\big(t_{k} - (\delta_{k} - \delta_{m})\ln{t_{k}}\big) - \frac{\abs{u}^{2}t_{m}}{t_{1} \cdots t_{m-1}} + \delta_{m} \ln{t_{m}}
   ,
  \end{equation}

  \begin{equation}\label{primativep2varying}
      {p}_{2}(t) = -\sum_{k=1}^{m-1}\big(t_{k} - (\delta_{k} - \delta_{m})\ln{t_{k}}\big) - \frac{\abs{u}^{2}t_{m}}{t_{1} \cdots t_{m-1}} + (1 + \delta_{m}) \ln{t_{m}},
    \end{equation}
    \begin{equation}\label{primativefvarying}
      {f}(t) = -\frac{u\overline{v'} + v\overline{u}}{\rho} \frac{t_{m}}{t_{1} \cdots t_{m-1}},
    \end{equation}
    and
    \begin{equation} \label{primativeqvarying}
      {q}(t) =   \prod_{k=1}^{m-1} t_{k}^{-1} e^{\frac{v\overline{v'}}{\rho^{2}} \frac{t_{m}}{t_{1} \cdots t_{m-1}}}.
    \end{equation}
 Here    $\gamma_{k}=\{t_k:\abs{t_k}= \tilde{s}_{k}, |t_k- \tilde{s}_{k}|<\varepsilon_{0}\}$, $k=1,\ldots,m$ for some $\varepsilon_{0}>0$. Moreover, both $\Re{p_{1}(t)}$ and $\Re{p_{2}(t)}$
 attain their unique minimum   over $\gamma_{1} \times \cdots \times \gamma_{m}$ at the point $  \tilde{s}  = ( \tilde{s}_{1}, \cdots,  \tilde{s}_{m-1},\tilde{s}_{m}).$
  \end{lem}

    \begin{proof} By Proposition \ref{repsumprop}  we have
  \begin{align}
      T_{N}(N^{m/2}z, N^{m/2}z')  &=   (2\pi i)^{-m}N^{-\sum_{k=1}^{m}a_{k}}(z\overline{z'})^{-a_m}\prod_{j=1}^{m}  \Gamma(a_{j} + 1)\nonumber \\ &\times \int_{\mathcal{C}_{1} \times \cdots \times \mathcal{C}_{m}}Q(t) e^{-\sqrt{N} {F}(t)} \frac{e^{-N{P}_{1}(t)} - e^{-N{P}_{2}(t)}}{t_1 \cdots t_{m}-1}\, d^mt \label{saddleform}
    \end{align}
    where
  \begin{equation}\label{primativeP1}
    {P}_{1}(t) = -\sum_{k=1}^{m-1}\big(t_{k} -  \delta_{k}\ln{t_{k}}\big)  -  \big(\abs{u}^{2}t_m-\delta_m \ln t_{m}\big),
  \end{equation}
  \begin{equation}\label{primativeP2}
    {P}_{2}(t) = -\sum_{k=1}^{m-1}\big(t_{k} -  \delta_{k}\ln{t_{k}}\big)-  \big(\abs{u}^{2}t_m-\delta_m \ln t_{m}\big) + \ln(t_1 \cdots t_{m}),
  \end{equation}
  and
  \begin{equation}\label{primativeQ}
    {F}(t) = -  \frac{u\overline{v'} + v\overline{u}}{\rho} t_{m}, \qquad Q(t) =  e^{\frac{v\overline{v'}}{\rho^{2}}  t_{m}}.
  \end{equation}
  Here    $\mathcal{C}_{k}$    is a path first going from   $-\infty$  to $r_{k}e^{i(-\pi+\theta_0)}$  ($0<\theta_0<\pi/2$) along the line parallel to the x-axis,  then going anticlockwise   along  the circle with radius of $r_k$  to $r_k e^{i(\pi-\theta_0)}$ and returning to  $-\infty$  along the line parallel to the x-axis. Choose $r_k=\tilde{s}_k$ for any $k<m$ and $r_m=1/(\tilde{s}_1 \cdots \tilde{s}_{m-1})$.

The subsequent procedure is the same as that in Lemma \ref{twopartslemma}.
  Let $D=\mathcal{C}_{1} \times \cdots \times \mathcal{C}_{m}$ and  $t_{0} = (\tilde{s}_1, \cdots, \tilde{s}_{m-1}, 1/(\tilde{s}_1 \cdots \tilde{s}_{m-1}))$. Note that  for all $t\in D$   the  following relations hold true
      \begin{equation}
        \Re{P_{1}(t)} > \Re{P_{1}(t_{0})}\  \mathrm{and}\   \Re{P_{2}(t)} > \Re{P_{2}(t_{0})} \ \mathrm{for}\ t \neq t_{0}.
      \end{equation}
     Moreover, by choosing a small $\varepsilon_{0}>0$, setting $$D_{\varepsilon_{0}}=\Big\{t\in D :\abs{t_1\cdots t_m-1}<\varepsilon_0; \abs{t_k-\tilde{s}_k}<\varepsilon_{0}, 1\leq k< m-1\Big\},$$ one can verify  that  there exists a positive number $c_0$ such that for sufficiently large $N$ and for every $t\in D \backslash D_{\varepsilon_0}$
     $$ \Re{P_1(t)-P_1(t_0)+\frac{F(t)-F(t_0)}{\sqrt{N}}} \geq 2 c_0$$
     and   $$ \Re{P_2(t)-P_2(t_0)+\frac{F(t)-F(t_0)}{\sqrt{N}}} \geq 2c_0.$$

     The remaining part is the same as that in Lemma \ref{twopartslemma} and then  the lemma immediately follows.
%
  \end{proof}

    \begin{lem}\label{vSumP1} With the same notation and assumptions as in Lemma \ref{twopartslemmavarying}, let
   $$\Upsilon =(z\overline{z'})^{a_{m}}\,\mathrm{P.V.} \int_{\gamma_{1} \times \cdots \gamma_{m}} {q}(t) e^{-\sqrt{N} {f}(t)} e^{-N{p}_{1}(t)} \frac{1}{t_{m}-1} d^{m}t,$$
    then the following asymptotics holds true uniformly for $v, v'$ in any compact subset of $\mathbb{C}$.
      \begin{enumerate}[label = $\mathrm{(\Roman*)}$]
        \item\label{vStatement1} For $\abs{u} \neq \sqrt{\delta_{1} \cdots \delta_{m}}$  and $u\neq 0$,
          \begin{equation*}
            \begin{split}
              &           \Upsilon\sim - i^{m}\pi \paren{\frac{2\pi}{N}}^{(m-1)/2} \frac{\mathrm{sign}\big( \delta_{m} - \abs{u}^{2} \xi(u)\big)}{\sqrt{1 + \abs{u}^{2}\xi(u) \sum_{k=1}^{m-1} \frac{1}{\tilde{s}_{k}}}} \sqrt{\xi(u)}  \\
              &\quad \times \exp\Big\{\xi(u)\frac{v\overline{v'}}{\rho^{2}}-\frac{1}{2} \xi^{2}(u) \frac{(u\overline{v'} + v\overline{u})^{2}}{\rho^{2}} \frac{\sum_{k=1}^{m-1} \frac{1}{\tilde{s}_{k}}}{1+\abs{u}^{2} \xi(u) \sum_{k=1}^{m-1} \frac{1}{\tilde{s}_{k}}}   \Big\}  \\
              &\quad \times \exp\Big\{\sqrt{N} \xi(u) \frac{u\overline{v'} + v\overline{u}}{\rho}+Nm \abs{u}^{2} \xi(u) + N\sum_{k=1}^{m-1} (\delta_{k} - \delta_{m}) \paren{1 - \ln \tilde{s}_{k}}\Big\}.
            \end{split}
          \end{equation*}
        \item\label{vStatement2} For  $\abs{u} = \sqrt{\delta_{1} \cdots \delta_{m}} \neq 0$,
          \begin{equation*}
            \begin{split}
              &\Upsilon\sim              i^{m}\pi \paren{\frac{2\pi}{N}}^{(m-1)/2} \frac{1}{\sqrt{1 + \abs{u}^{2} \xi(u) \sum_{k=1}^{m-1} \frac{1}{\tilde{s}_{k}}}}\sqrt{\xi(u)} \\
              &\quad \times \exp\Big\{\xi(u)\frac{v\overline{v'}}{\rho^{2}}-\frac{1}{2} \xi^{2}(u) \frac{(u\overline{v'} + v\overline{u})^{2}}{\rho^{2}} \frac{\sum_{k=1}^{m-1} \frac{1}{\tilde{s}_{k}}}{1+\abs{u}^{2} \xi(u) \sum_{k=1}^{m-1} \frac{1}{\tilde{s}_{k}}}   \Big\}  \\
              &\quad \times \exp\Big\{\sqrt{N} \xi(u) \frac{u\overline{v'} + v\overline{u}}{\rho}+Nm \abs{u}^{2} \xi(u) + N\sum_{k=1}^{m-1} (\delta_{k} - \delta_{m}) \paren{1 - \ln \tilde{s}_{k}}\Big\}\\
                            & \quad \times \mathrm{erf}\Big(\frac{\sqrt{\xi(u)}}{\sqrt{1+\abs{u}^{2} \xi(u) \sum_{k=1}^{m-1} \frac{1}{\tilde{s}_{k}}}} \frac{u\overline{v'} + v\overline{u}}{\rho\sqrt{2\delta_{1} \cdots \delta_{m}}}\Big).
            \end{split}
          \end{equation*}
      \end{enumerate}
    \end{lem}
    \begin{proof}
     Let $\tilde{s}_{k}$ be defined in \eqref{tildes-constants}, write  $\tilde{s} = (\tilde{s}_{1}, \cdots, \tilde{s}_{m})$. Then we know that   the following inequality
             \begin{equation}
        \Re{p_{1}(t)} > \Re{p_{1}(\tilde{s})}
      \end{equation}
    holds   for all $t \in \gamma_{1} \times \cdots \times \gamma_{m}$ except for   $\tilde{s}$, and also  \begin{equation*}
        p_{1}(\tilde{s}) = -m \abs{u}^{2} \xi(u) - \sum_{k=1}^{m-1} (\delta_{k} - \delta_{m}) \paren{1 - \ln \tilde{s}_{k}}.
      \end{equation*}

      Notice
      \begin{equation}
        \frac{\partial p_{1}}{\partial t_{k}} = \begin{cases}
          -1 + (\delta_{k} - \delta_{m}) \frac{1}{t_{k}} + \frac{\abs{u}^{2}t_{m}}{t_{1} \cdots t_{m-1} t_{k}},& k <m,\\
          -\frac{\abs{u}^{2}}{t_{1} \cdots t_{m-1}} + \delta_{m} \frac{1}{t_{m}}, & k=m,
        \end{cases}
      \end{equation}
we have       \begin{equation*}
        \frac{\partial p_{1}(\tilde{s})}{\partial t_{k}} = 0, \quad  k=1, \cdots, m-1.
      \end{equation*}

     On the other hand, for the Hessian matrix
      \begin{equation*}
        A = \paren{\frac{\partial^{2} p_{1} (\tilde{s})}{\partial t_{k} \partial t_{j}}}_{1 \leq k, j \leq m-1},
      \end{equation*}
      where
      \begin{equation*}
        \frac{\partial^{2} p_{1} (\tilde{s})}{\partial t_{k} \partial t_{j}} = \begin{cases}
          - \abs{u}^{2} \xi(u) \frac{1}{\tilde{s}_{k} \tilde{s}_{j}}, & k \neq j,\\
          -\frac{1}{\tilde{s}_{k}} - \abs{u}^{2} \xi(u) \frac{1}{\tilde{s}_{k}^{2}}, & k = j,
        \end{cases}
      \end{equation*}
      simple manipulation gives the determinant
      \begin{equation}
        \Det{A} = (-1)^{m-1} \xi(u) \big(1 + \xi(u) \abs{u}^{2} \sum_{k=1}^{m-1} \frac{1}{\tilde{s}_{k}}\big)
      \end{equation}
      and the inverse
      \begin{equation}
        A^{-1} = -\mathrm{diag}(\tilde{s}_{1}, \cdots, \tilde{s}_{m-1}) + \frac{\abs{u}^{2} \xi(u)}{1+\abs{u}^{2} \xi(u) \sum_{k = 1}^{m-1} \frac{1}{\tilde{s}_{k}}} 1_{(m-1) \times (m-1)},
      \end{equation}
      where $1_{(m-1) \times (m-1)}$ is the $(m-1) \times (m-1)$ matrix with all entries $1$.
      Hence if letting    $\eta = (\frac{\partial f (\tilde{s})}{\partial t_{1}}, \cdots, \frac{\partial f (\tilde{s})}{\partial t_{m-1}})$ where
      \begin{equation}
        \frac{\partial f (\tilde{s})}{\partial t_{k}} = \xi(u) \frac{u\overline{v'} + v\overline{u}}{\rho} \frac{1}{\tilde{s}_{k}},
      \end{equation}
     then
      \begin{equation}
        \eta A^{-1} \eta^{T} = - \xi^{2}(u) \frac{(u\overline{v'} + v\overline{u})^{2}}{\rho^{2}} \frac{1}{1+\abs{u}^{2} \xi(u) \sum_{k=1}^{m-1} \frac{1}{\tilde{s}_{k}}} \sum_{k=1}^{m-1} \frac{1}{\tilde{s}_{k}}.
      \end{equation}

       When $\abs{u} \neq \sqrt{\delta_{1} \cdots \delta_{m}}$  and $u\neq 0$, let $\alpha: = \frac{\partial p_{1}(\tilde{s})}{\partial t_{m}} = \delta_{m} - \abs{u}^{2} \xi(u)$, application of  Theorem \ref{laplaceMethod1} shows
      \begin{equation*}
        \begin{split}
          \Upsilon &\sim (-1)^{m} i\pi \paren{\frac{2\pi}{N}}^{(m-1)/2} \xi(u)\, e^{\xi(u)\frac{v\overline{v'}}{\rho^{2}}} e^{\sqrt{N} \xi(u) \frac{u\overline{v'} + v\overline{u}}{\rho}}\\
          &\quad \times e^{Nm \abs{u}^{2} \xi(u) + N\sum_{k=1}^{m-1} (\delta_{k} - \delta_{m}) \paren{1 - \ln s_{k}}} e^{\frac{1}{2} \eta A^{-1} \eta^{T}} \frac{\mathrm{sign}(\Re{\alpha})}{\sqrt{\Det{A}}}
        \end{split}
      \end{equation*}
     from which Statement \ref{vStatement1} follows.

      When $\abs{u} = \sqrt{\delta_{1} \cdots \delta_{m}} \neq 0$, we then have
      \begin{equation*}
        \frac{\partial p_{1}(\tilde{s})}{\partial t_{k}} = 0, \quad, k=1, \cdots, m.
      \end{equation*}
      Write       \begin{equation*}
        \beta: = (\frac{\partial^{2} p_{1}(\tilde{s})}{\partial t_{1} \partial t_{m}}, \cdots, \frac{\partial^{2} p_{1}(\tilde{s})}{\partial t_{m-1} \partial t_{m}}) = \abs{u}^{2} \xi(u) (\frac{1}{\tilde{s}_{1}}, \cdots, \frac{1}{\tilde{s}_{m-1}}),
      \end{equation*}
      and
      \begin{equation*}
        \alpha': = \frac{\partial^{2} p_{1}(\tilde{s})}{\partial t_{m}^{2}} = -\delta_{m}, \qquad
        \eta_{m}: = \frac{\partial f(\tilde{s})}{\partial t_{m}} = - \xi(u) \frac{u\overline{v'} + v\overline{u}}{\rho}.
      \end{equation*}
      Following  Theorem \ref{steepestDecentonEdge}, we obtain
      \begin{equation*}
        \begin{split}
         \Upsilon
          &\sim (-1)^{m-1} \paren{\frac{2\pi}{N}}^{(m-1)/2} \xi(u) e^{\xi(u)\frac{v\overline{v'}}{\rho^{2}}} e^{\sqrt{N} \xi(u) \frac{u\overline{v'} + v\overline{u}}{\rho}} e^{\frac{1}{2} \eta A^{-1} \eta^{T}} \frac{\mathrm{sign}(\Re{\alpha'})}{\sqrt{\Det{A}}}\\
          &  \times e^{Nm \abs{u}^{2} \xi(u) + N\sum_{k=1}^{m-1} (\delta_{k} - \delta_{m}) \paren{1 - \ln \tilde{s}_{k}}}  i \pi \,
           \mathrm{erf}\big( \frac{
        i\eta_{m} - i\beta A^{-1}  \eta^{T}
        }{ \sqrt{2\alpha' - 2\beta A^{-1} \beta^{T}}
        }\big)
        \end{split}
      \end{equation*}
     from which Statement \ref{vStatement2} follows.

     Obviously,   the asymptotics holds true  uniformly for $v, v'$ in any compact subset of $\mathbb{C}$.
    \end{proof}

    Similar calculations as in the proof of Lemma \ref{vSumP1} may afford us the  asymptotics corresponding to the function $p_{2}(t)$.
    \begin{lem}\label{vSumP2}
      With the same notation and assumptions as in Lemma \ref{twopartslemmavarying}, let
   $$\Phi =(z\overline{z'})^{a_{m}}\,\mathrm{P.V.} \int_{\gamma_{1} \times \cdots \gamma_{m}} {q}(t) e^{-\sqrt{N} {f}(t)} e^{-N{p}_{2}(t)} \frac{1}{t_{m}-1} d^{m}t,$$
    then the following asymptotics holds true uniformly for $v, v'$ in any compact subset of $\mathbb{C}$.
      \begin{enumerate}[label = $\mathrm{(\Roman*)}$]
        \item\label{vStatement3} For  $\abs{u} \neq \sqrt{(1+\delta_{1}) \cdots (1+\delta_{m})}$ and $u\neq 0$,
          \begin{equation*}
            \begin{split}
              & \Phi \sim
              - i^{m}\pi \paren{\frac{2\pi}{N}}^{(m-1)/2} \frac{\mathrm{sign}\big( 1+\delta_{m} - \abs{u}^{2} \xi(u)\big)}{\sqrt{1 + \abs{u}^{2}\xi(u) \sum_{k=1}^{m-1} \frac{1}{\tilde{s}_{k}}}} \sqrt{\xi(u)}  \\
              &\quad \times \exp\Big\{\xi(u)\frac{v\overline{v'}}{\rho^{2}}-\frac{1}{2} \xi^{2}(u) \frac{(u\overline{v'} + v\overline{u})^{2}}{\rho^{2}} \frac{\sum_{k=1}^{m-1} \frac{1}{\tilde{s}_{k}}}{1+\abs{u}^{2} \xi(u) \sum_{k=1}^{m-1} \frac{1}{\tilde{s}_{k}}}   \Big\}  \\
              &\quad \times \exp\Big\{\sqrt{N} \xi(u) \frac{u\overline{v'} + v\overline{u}}{\rho}+Nm \abs{u}^{2} \xi(u) + N\sum_{k=1}^{m-1} (\delta_{k} - \delta_{m}) \paren{1 - \ln \tilde{s}_{k}}\Big\}.
            \end{split}
          \end{equation*}
        \item\label{vStatement4} For $\abs{u} = \sqrt{(1+\delta_{1}) \cdots (1+\delta_{m})}$,
          \begin{equation*}
            \begin{split}
              &\Phi\sim    i^{m}\pi \paren{\frac{2\pi}{N}}^{(m-1)/2} \frac{1}{\sqrt{1 + \abs{u}^{2}\xi(u) \sum_{k=1}^{m-1} \frac{1}{\tilde{s}_{k}}}}\sqrt{\xi(u)} \\
              &\quad \times \exp\Big\{\xi(u)\frac{v\overline{v'}}{\rho^{2}}-\frac{1}{2} \xi^{2}(u) \frac{(u\overline{v'} + v\overline{u})^{2}}{\rho^{2}} \frac{\sum_{k=1}^{m-1} \frac{1}{\tilde{s}_{k}}}{1+\abs{u}^{2} \xi(u) \sum_{k=1}^{m-1} \frac{1}{\tilde{s}_{k}}}   \Big\}  \\
              &\quad \times \exp\Big\{\sqrt{N} \xi(u) \frac{u\overline{v'} + v\overline{u}}{\rho}+Nm \abs{u}^{2} \xi(u) + N\sum_{k=1}^{m-1} (\delta_{k} - \delta_{m}) \paren{1 - \ln \tilde{s}_{k}}\Big\}\\
                            & \quad \times \mathrm{erf}\Big(\frac{\sqrt{\xi(u)}}{\sqrt{1+\abs{u}^{2} \xi(u) \sum_{k=1}^{m-1} \frac{1}{\tilde{s}_{k}}}} \frac{u\overline{v'} + v\overline{u}}{\rho\sqrt{2(1+\delta_{1}) \cdots (1+\delta_{m})}}\Big).
            \end{split}
          \end{equation*}
      \end{enumerate}
    \end{lem}
\subsection{Proofs of Theorems \ref{limitdensityvarying}  and \ref{univeralitytheoremvarying}}
    We are now  ready to prove the limiting eigenvalue density and local universality in the  parameter-varying case.

    \begin{proof}[Proof of Theorems \ref{limitdensityvarying}]
      Taking $v = 0$ in Lemma \ref{vWeight}, we have

      \begin{equation}
        \begin{split}
          w_{m}(N^{m/2}z) &\sim \sqrt{\xi(u)} \abs{z}^{2a_{m}} \paren{\frac{2\pi}{N}}^{(m-1)/2}\prod_{k=1}^{m} \frac{N^{a_{k}}}{\Gamma(a_{k}   + 1)}\\
              & \times \frac{1}{\pi \sqrt{1 + \abs{u}^{2}\xi(u) \sum_{k=1}^{m-1} \frac{1}{\tilde{s}_{k}}}}\exp\big\{-N m\abs{u}^{2}\xi(u)\big\} \\
              &  \times \exp\Big\{ -N \sum_{k=1}^{m-1} (\delta_{k} - \delta_{m}) \big(1 - \ln (\delta_{k} - \delta_{m} + \abs{u}^{2} \xi(u) )\big)\Big\}. \label{varyingdensity1}
                     \end{split}
      \end{equation}

      Taking $v, v' = 0$ in Lemmas \ref{twopartslemmavarying}, \ref{vSumP1} and \ref{vSumP2}, we get the large $N$ asymptotics of  \eqref{decayandleadingvarying}
      \begin{equation}
        \begin{split}
          T_{N}(N^{m/2}z, &N^{m/2}z)  \sim  \pi \abs{z}^{-2a_{m}} \sqrt{\xi(u)} \paren{\frac{2\pi}{N}}^{(m-1)/2} \prod_{k=1}^{m} \frac{\Gamma(a_{k} N + 1)}{2\pi N^{a_{k} N}}\\
          &  \times  \exp\Big\{ N\sum_{k=1}^{m-1} (\delta_{k} - \delta_{m})
           \big(1 - \ln (\delta_{k} - \delta_{m} + \abs{u}^{2} \xi(u) )\big)
          \Big\}\\
          & \times \exp\big\{Nm \abs{u}^{2} \xi(u)\big\} \frac{\chi_{\set{z:\sqrt{\delta_{1} \cdots \delta_{m}} < \abs{z} < \sqrt{(1+\delta_{1}) \cdots (1+\delta_{m})}}}}{\sqrt{1 + \abs{u}^{2}\xi(u) \sum_{k=1}^{m-1} \frac{1}{\tilde{s}_{k}}}}. \label{varyingdensity2}
        \end{split}
      \end{equation}

      Since
        $R_{N,1}(N^{m/2}z)=   w_{m}(N^{m/2}z) T_{N}(N^{m/2}z, N^{m/2}z)$ and $z = u$, combination of \eqref{varyingdensity1} and \eqref{varyingdensity2} completes the proof.
    \end{proof}

    \begin{proof}[Proof of Theorem \ref{univeralitytheoremvarying}]
      Let
      \begin{align*}
          \psi_{N}(v) &= \exp\Big\{
            \delta_{m} \sqrt{N} \frac{u\overline{v} - v\overline{u}}{2\abs{u}^{2} \rho} -  \sqrt{N} \xi_{m}(u) \frac{u\overline{v} - v\overline{u}}{2\rho}
           -  \delta_{m} \frac{(u\overline{v})^{2} - (v\overline{u})^{2}}{4\abs{u}^{4} \rho^{2}} \Big\}\\
          &  \times \exp\Big\{\frac{1}{4} \xi_{m}^{2}(u) \frac{(u\overline{v})^{2} - (v\overline{u})^{2}}{\rho^{2}} \frac{\sum_{k=1}^{m-1} \frac{1}{\delta_{k} - \delta_{m} + \abs{u}^{2} \xi_{m}(u)}}{1+\abs{u}^{2} \xi(u) \sum_{k=1}^{m-1} \frac{1}{\delta_{k} - \delta_{m} + \abs{u}^{2} \xi_{m}(u)}}\Big\},
        \end{align*}
      and let the diagonal matrix   $D = \mathrm{diag}(\psi_{N}(v_{1}), \cdots, \psi_{N}(v_{n}))$.

      In the bulk,    combining    Lemma \ref{vWeight}, Lemma \ref{twopartslemmavarying}, Statement \ref{vStatement1} in Lemma \ref{vSumP1} and  \ref{vStatement3} in Lemma \ref{vSumP2}, we obtain
      \begin{equation*}
               \psi_{N}(v_{k}) K_{N}(N^{m/2}z_{k},N^{m/2}z_{j}) \psi_{N}^{-1}(v_{j})          \sim N^{-(m-1)} \frac{1}{\pi} \rho^{2}
            e^{-\frac{1}{2} (\abs{v_{k}}^{2} + \abs{v_{j}}^{2} - 2v_{k}\overline{v_{j}})}.
             \end{equation*}
     Furthermore,
      \begin{equation*}
        \begin{split}
          R_{N,n}(N^{m/2}z_1, \ldots,N^{m/2}z_n)& = \Det{D(K_{N}(N^{m/2}z_{k}, N^{m/2}z_{j}))D^{-1}}\\
          &\sim N^{-n(m-1)} {\rho^{2n}} \Det{\frac{1}{\pi} e^{-\frac{1}{2} (\abs{v_{k}}^{2} + \abs{v_{j}}^{2} - 2v_{k}\overline{v_{j}})}}_{1 \leq k, j \leq n}.
        \end{split}
      \end{equation*}
     This completes the bulk limit of Theorem   \ref{univeralitytheoremvarying}.

     Likewise, combining    Lemmas \ref{vWeight},   \ref{twopartslemmavarying},  \ref{vSumP1} and    \ref{vSumP2} we can prove the inner and outer edge cases.
    \end{proof}

 \section{Products of truncated unitary matrices} \label{truncatedproductsection}
 We take  the same procedure as    in Section \ref{varyingginibre} to tackle the product of independent  truncated unitary matrices  $X^{(m)} = X_{m} X_{m-1} \cdots X_{1}$  where each $X_k$ has the joint probability  density  proportional to \begin{equation} \big(\Det{X^{*}_{k}X_{k}}\big)^{a_k} \big(\Det{I_{N}-X^{*}_{k}X_{k}}\big)^{L_k-N}1_{\{  I_{N}-X^{*}_{k}X_{k}>0 \}},   \end{equation}
  where all $a_k>-1$ and $L_k\geq N$. 
   The joint eigenvalue density function of $X^{(m)}$ reads off
  \begin{equation}
    P_{N}(z_1,\ldots,z_N)=   \frac{1}{N!} \Det{K_{N}(z_{j}, z_{k})}_{1 \leq j, k \leq N}
  \end{equation}
   where the correlation kernel  equals to
  \begin{equation}\label{kernelfortruncated}
    K_{N}(z, z') = \sqrt{w_{m}(z)w_{m}(\overline{z'})}\,T_{N}(z,z')
  \end{equation}
  with all $\abs{z_j}<1$; as a matter of fact, the density (5.2) is applicable for all $L_k>0$ and includes the  product of  truncated rectangular matrices, see \cite{adhikari2013determinantal,akemann2014universal,ipsen2014weak}.  Further, the $n$-point  correlation functions can be expressed as
  \begin{align}
    R_{N,n}(z_{1}, \cdots, z_{n}) =\Det{K_{N}(z_{j}, z_{k})}_{1 \leq j, k \leq n}.\label{correlationfunctionfortruncated}
  \end{align}
  If letting $a_k=\sigma_{k}N + b_{k}$ and $L_k=\tau_k N$ with $\sigma_{1}, \ldots, \sigma_{m}\geq 0$ and $\tau_{1}, \ldots, \tau_{m}> 0$,
   the weight function can be written as the   integral (cf. Section 2.2 \cite{akemann2014universal} and Eq. (99) \cite{ipsen2014weak})
  \begin{multline}\label{weightTr}
    w_{m}(z) = \frac{2^{m-1} \abs{z}^{2\sigma_{m}N+2b_m}}{\pi} \prod_{k=1}^{m} \frac{\Gamma((\sigma_{k} + \tau_{k})N + b_{k} + 2)}{\Gamma(\sigma_{k} N + b_{k} + 1) \Gamma(\tau_{k}N + 1)}\\ \times \int_{[0,1]^{m-1}} e^{-Np(r)} q(r) d^{m-1}r,
  \end{multline}
  where
  \begin{equation}\label{pforWeightTr}
    \begin{split}
      p(r) &= - \sum_{k=1}^{m-1} \big(2(\sigma_{k} - \sigma_{m})\ln{r_{k}} + \tau_{k} \ln(1-r_{k}^{2})_{+}\big) - \tau_{m} \ln{\Big(1-\frac{\abs{z}^{2}}{\prod_{k=1}^{m-1} r_{k}^{2}}\Big)_{+}},
    \end{split}
  \end{equation}
  and
  \begin{equation}
    q(r) = \prod_{k=1}^{m-1} r_{k}^{2(b_{k}-b_m)-1}.
  \end{equation}
Also, the truncated series  equals to
     \begin{equation}\label{finitesumTr}
       \begin{split}
         T_{N}(z,z') &= \prod_{k=1}^{m} \frac{\Gamma(\sigma_{k}N + b_{k} + 1)}{\Gamma((\sigma_{k} + \tau_{k}) N + b_{k} + 2)} \sum_{l=0}^{N-1} \prod_{k=1}^{m} \frac{\Gamma((\sigma_{k} + \tau_{k}) N + b_{k} + l + 2)}{\Gamma(\sigma_{k} N + b_{k} + l + 1)} (z \overline{z'})^{l}.
       \end{split}
     \end{equation}

   \subsection{Integral representations}\label{repsumpropTr}
     As in the Ginibre case  we first rewrite the truncated series $T_{N}(z,z')$  as  multivariate  integrals.
     \begin{prop}
       For a fixed nonzero complex number $u$, let $\rho = \rho(u) \neq 0$. Introduce the rescaling variables
       \begin{equation}\label{rescalingzTr}
         z = u + \frac{v}{\rho \sqrt{N}}, \quad z' = u + \frac{v'}{\rho \sqrt{N}},
       \end{equation}
       where $v$ and $v'$  lie in a compact set of $\Real{C}$. Then
       \begin{align}
           T_{N}(z,z')&= N^{2m + \sum_{k=1}^{m} \tau_{k}} (z\overline{z'})^{-(\sigma_{m}N + b_{m})} (2\pi i)^{-m} \prod_{k=1}^{m} \frac{\Gamma(\sigma_{k}N + b_{k} + 1)}{\Gamma((\sigma_{k} + \tau_{k})N + b_{k} + 2)} \nonumber\\
           &\quad \times \int_{(0,\infty)^{m} \times \mathcal{C}_{1} \times \cdots \times \mathcal{C}_{m}} Q(t) e^{-\sqrt{N}F(t)} \frac{e^{-NP_{1}(t)} - e^{-NP_{2}(t)}}{t_{m+1} \cdots t_{2m} - 1} d^{2m}t.
                \end{align}
       where
       \begin{align}
                    P_{1}(t) &= \sum_{k=1}^{m} \big(t_{k} - (\sigma_{k} - \sigma_{m} + \tau_{k}) \ln{t_{k}}\big) - \sum_{k=1}^{m-1} (t_{m+k} - \sigma_{k} \ln{t_{m+k}})\nonumber\\
           &\quad  - (t_{1} \cdots t_{m}) \abs{u}^{2} t_{2m} + \sigma_{m} \ln{t_{2m}},
               \end{align}
       \begin{equation}
                P_{2}(t) =P_{1}(t) +   \ln{(t_{m+1}\cdots t_{2m})},
       \end{equation}
       \begin{equation}
           Q(t) = \prod_{k=1}^{m} t_{k}^{b_{k} - b_{m} + 1} \prod_{k=1}^{m} t_{m+k}^{-b_{k}} e^{t_{2m} t_{1} \cdots t_{m} \frac{v\overline{v'}}{\rho^{2}}},
       \end{equation}
       \begin{equation}
           F(t) = -t_{2m} t_{1} \cdots t_{m} \frac{u\overline{v'} + v\overline{u}}{\rho}.
       \end{equation}
        Here    $\mathcal{C}_{k}$    is a path first going from   $-\infty$ to $r_{k}e^{i(-\pi+\theta_0)}$  ($0<\theta_0<\pi/2$) along the line parallel to the x-axis,  then going anticlockwise   along  the circle with radius of $r_k$  to $r_k e^{i(\pi-\theta_0)}$ and returning to   $-\infty$  along the line parallel to the x-axis. The radius $r_k$ will be chosen properly as required.
     \end{prop}
     \begin{proof}
       Using   integral  representations of  reciprocal Gamma functions (\ref{intrepgamma1}) and Gamma functions, we find
       \begin{align*}
           T_{N}(z, z') &= \prod_{k=1}^{m} \frac{\Gamma(\sigma_{k}N + b_{k} + 1)}{\Gamma((\sigma_{k} + \tau_{k})N + b_{k} + 2)} \sum_{l=0}^{N-1} (z\overline{z'})^{l} \prod_{k=1}^{m} \int_{(0,\infty)} t_{k}^{(\sigma_{k} + \tau_{k}) N + b_{k} + l + 1} e^{-t_{k}} dt_{k} \\
           &\quad \times \prod_{k=1}^{m} \frac{1}{2\pi i} \int_{\mathcal{C}_{k}} t_{m+k}^{-(\sigma_{k}N + b_{k} + l + 1)} e^{t_{m+k}} dt_{m+k} \\
           &= (2\pi i)^{-m} \prod_{k=1}^{m} \frac{\Gamma(\sigma_{k}N + b_{k} + 1)}{\Gamma((\sigma_{k} + \tau_{k})N + b_{k} + 2)} \int_{(0,\infty)^{m} \times \mathcal{C}_{1} \times \cdots \times \mathcal{C}_{m}} \frac{1 - \paren{\frac{t_{1} \cdots t_{m} z\overline{z'}}{t_{m+1} \cdots t_{2m}}}^{N}}{1 - \frac{t_{1} \cdots t_{m} z\overline{z'}}{t_{m+1} \cdots t_{2m}}} \\
           &\quad \times \prod_{k=1}^{m} \big(t_{k}^{(\sigma_{k} + \tau_{k}) N + b_{k} + 1} t_{m+k}^{-(\sigma_{k}N + b_{k} + 1)} e^{-t_{k} + t_{m+k}}\big) d^{2m} t.
              \end{align*}
       Since
       \begin{equation*}
         z\overline{z'} = \abs{u}^{2} + \frac{u\overline{v'} + v\overline{u}}{\rho \sqrt{N}} + \frac{v\overline{v'}}{\rho^{2} N},
       \end{equation*}
       taking change of variables $t_{k} \to N t_{k}$, $k =1, \cdots, 2m-1$, and $t_{2m} \to N t_{2m} t_{1} \cdots t_{m} z\overline{z'}$ for the large $N$, we thus obtain the claimed result.
     \end{proof}

   \subsection{Several lemmas}
     Without loss of generality, we assume that
     \begin{equation}
       \sigma_{m} = \min\set{\sigma_{k}: k = 1, \cdots, m}.
     \end{equation}
     For  any  $0<\abs{u} < 1$, we know that there exists a unique root  in the interval $(\abs{u}, 1)$ for the algebraic equation
     \begin{multline}\label{algebraicEqTr1}
     x\prod_{k=1}^{m-1} \paren{(\tau_{k} + \sigma_{k} - \sigma_{m}) (x - \abs{u}^{2}) + \tau_{m} \abs{u}^{2}} \\
  \quad - \prod_{k=1}^{m-1} \paren{(\sigma_{k} - \sigma_{m}) (x - \abs{u}^{2}) + \tau_{m} \abs{u}^{2}} =0.
           \end{multline}
   Let $\xi(u)$ be such a root.

     Now the asymptotics for the weight function can be stated as follows.
     \begin{lem}\label{TrWeight}
       Let
       \begin{equation}\label{s-constantsTr}
         s_{k} = \paren{\frac{(\sigma_{k} - \sigma_{m}) (\xi(u) - \abs{u}^{2}) + \tau_{m} \abs{u}^{2}}{(\tau_{k} + \sigma_{k} - \sigma_{m}) (\xi(u) - \abs{u}^{2}) + \tau_{m} \abs{u}^{2}}}^{1/2}, \quad k = 1, \cdots, m-1.
       \end{equation}
       and introduce a rescaling variable
       \begin{equation}
         z = u + \frac{v}{\rho \sqrt{N}},
       \end{equation}
       where $\rho$ is a real nonzero parameter depending on $u$. Then for $0<\abs{u} < 1$ as $N \to \infty$ we have
             \begin{align*}
           w_{m}(z) &\sim  \Big(\frac{2\pi}{N}\Big)^{(m-1)/2} \prod_{k=1}^{m} \frac{\Gamma((\sigma_{k} + \tau_{k})N + b_{k} + 2)}{\Gamma(\sigma_{k} N + b_{k} + 1) \Gamma(\tau_{k}N + 1)}\, \abs{z}^{2(\sigma_{m}N + b_{m})}     \\
           &  \times   \frac{1}{\pi \sqrt{ \Big(1 + \frac{\tau_{m} \abs{u}^{2} \xi(u)}{(\xi(u) - \abs{u}^{2})^{2}}  \sum_{k=1}^{m-1} \frac{(1-s_{k}^{2})^{2}}{\tau_{k} s_{k}^{2} }\Big)\prod_{k=1}^{m-1} \frac{\tau_{k}}{(1-s_{k}^{2})^{2}}}} \prod_{k=1}^{m-1} s_{k}^{2(b_{k} - b_{m})-1}\\
           & \times \mathrm{exp}\Big\{N\tau_{m} \ln\Big(1-\frac{1}{\xi(u) - \abs{u}^{2}} \big(\frac{u\overline{v} + v\overline{u}}{\rho \sqrt{N}} + \frac{\abs{v}^{2}}{\rho^{2} N}\big)\Big)+N\tau_{m} \ln{\big(1-\frac{\abs{u}^{2}}{\xi(u)}\big)}\Big\} \\
           &  \times \mathrm{exp}\Big\{N\sum_{k=1}^{m-1} \big(2(\sigma_{k} - \sigma_{m})\ln{s_{k}} + \tau_{k} \ln(1-s_{k}^{2})\big)\Big\} \\
           &  \times \mathrm{exp}\Big\{\frac{1}{2}\frac{\tau_{m}^{2} \xi^{2}(u)}{(\xi(u) - \abs{u}^{2})^{4}} \frac{(u\overline{v}+v\overline{u})^{2}}{\rho^{2}} \frac{ \sum_{k=1}^{m-1} \frac{(1-s_{k}^{2})^{2}}{\tau_{k} s_{k}^{2} }}{1 + \frac{\tau_{m} \abs{u}^{2} \xi(u)}{(\xi(u) - \abs{u}^{2})^{2}} \sum_{k=1}^{m-1} \frac{(1-s_{k}^{2})^{2}}{\tau_{k} s_{k}^{2}}}\Big\}.
         \end{align*}

     \end{lem}
     \begin{proof}
       For $0<\abs{u} < 1$, if the variable $z$ in  the function $p(r)$ of (\ref{pforWeightTr})  is replaced by $u$,  we can verify   that $s = (s_{1}, \cdots, s_{m-1})$ is its unique saddle point and is also a minimum point.  Thus it suffices for us to analyze the behavior near the neighborhood of $s$. Noting \eqref{algebraicEqTr1} and \eqref{s-constantsTr}, we always choose a small   neighborhood of $s$ such that
       \begin{equation*}
         1 - \frac{\abs{u}^{2}}{r_{1}^{2} \cdots r_{m-1}^{2}} > 0
       \end{equation*}
       Furthermore, by setting $z = u + \frac{v}{\rho \sqrt{N}}$
       one gets  for $N$  large enough
       \begin{equation}
         \begin{split}
           & \prod_{k=1}^{m} \frac{\Gamma(\sigma_{k} N + b_{k} + 1) \Gamma(\tau_{k}N + 1)}{\Gamma((\sigma_{k} + \tau_{k})N + b_{k} + 2)}\  w_{m}(z)\\
           &= \frac{2^{m-1} \abs{z}^{2(\sigma_{m}N + b_{m})}}{\pi} \int_{[0,1]^{m-1}} q(r) e^{-\sqrt{N} f_{N}(r)}e^{-Np(r)} d^{m-1}r,
         \end{split}
       \end{equation}
       where (for simplicity we use the same notation $ p(r)$ as that in \eqref{pforWeightTr})
       \begin{equation}
         p(r) = - \sum_{k=1}^{m-1} \big(2(\sigma_{k} - \sigma_{m})\ln{r_{k}} + \tau_{k} \ln(1-r_{k}^{2} )_{+}\big) - \tau_{m} \ln{\big(1-\frac{\abs{u}^{2}}{\prod_{k=1}^{m-1} r_{k}^{2}}\big)_{+}} ,
       \end{equation}
       \begin{equation}
         q(r) = \prod_{k=1}^{m-1} r_{k}^{2(b_{k}-b_m)-1},
       \end{equation}
       and
       \begin{equation}
         f_{N}(r) = -\sqrt{N} \tau_{m} \ln\Big(1-\frac{1}{r_{1}^{2} \cdots r_{m-1}^{2} - \abs{u}^{2}} \big(\frac{u\overline{v} + v\overline{u}}{\rho \sqrt{N}} + \frac{\abs{v}^{2}}{\rho^{2} N}\big)\Big).
       \end{equation}

Next, we do some explicit calculations as follows.  First, the Hessian matrix of $p(t)$ at $s$  reads off
       \begin{equation}
         A = \paren{\frac{\partial^{2} p(s)}{\partial r_{k} \partial r_{j}}}_{1\leq j,k\leq m-1},
       \end{equation}
       where
       \begin{equation}
         \frac{\partial^2 p(s)}{\partial r_{k} \partial r_{j}} = \begin{cases}
           4 \tau_{m} \frac{1}{s_{k} s_{j}} \frac{\abs{u}^{2} \xi(u)}{(\xi(u) - \abs{u}^{2})^{2}}, & k \neq j;\\
           4 \tau_{k} \frac{1}{(1-s_{k}^{2})^{2}} + 4 \tau_{m} \frac{1}{s_{k}^{2}} \frac{\abs{u}^{2} \xi(u)}{(\xi(u) - \abs{u}^{2})^{2}}, & k =j.
         \end{cases}
       \end{equation}
       It is a positive definite matrix and   
      the inverse is equal to $A^{-1} = (A_{kj}^{*})$  where for $k\neq j$
      \begin{equation}
         A_{kj}^{*}  =
           -  \frac{\frac{(1-s_{k}^{2})^{2}}{ \tau_{k}s_{k}} \frac{(1-s_{j}^{2})^{2}}{ \tau_{j}s_{j} } \frac{\tau_{m} \abs{u}^{2} \xi(u)}{(\xi(u) - \abs{u}^{2})^{2}}}{1 + \frac{\tau_{m} \abs{u}^{2} \xi(u)}{(\xi(u) - \abs{u}^{2})^{2}} \sum_{l = 1}^{m-1} \frac{(1-s_{l}^{2})^{2}}{\tau_{l}s_{l}^{2}}}\end{equation}
                     and for $k=j$ \begin{equation}
           A_{kk}^{*}= \frac{(1-s_{k}^{2})^{2}}{\tau_{k}} \frac{1 + \frac{\tau_{m} \abs{u}^{2} \xi(u)}{(\xi(u) - \abs{u}^{2})^{2}} \sum_{l \neq k} \frac{(1-s_{l}^{2})^{2}}{\tau_{l}s_{l}^{2} }}{1 + \frac{\tau_{m} \abs{u}^{2} \xi(u)}{(\xi(u) - \abs{u}^{2})^{2}} \sum_{l = 1}^{m-1} \frac{(1-s_{l}^{2})^{2}}{\tau_{l}s_{l}^{2} }}.
                \end{equation}

Secondly, the determinant of the Hessian at $s$ is equal to
       \begin{equation}
         \Det{A} = 4^{m-1}  \Big(1 + \frac{\tau_{m} \abs{u}^{2} \xi(u)}{(\xi(u) - \abs{u}^{2})^{2}} \sum_{k=1}^{m-1} \frac{(1-s_{k}^{2})^{2}}{\tau_{k}s_{k}^{2} }\Big) \prod_{k=1}^{m-1} \frac{\tau_{k}}{(1-s_{k}^{2})^{2}}
       \end{equation}
and        \begin{equation}
         \eta := \lim_{N \to \infty} \nabla f_{N}(s) = -2 \frac{\tau_{m} \xi(u)}{(\xi(u) - \abs{u}^{2})^{2}} \frac{u\overline{v} + v\overline{u}}{\rho} \big(\frac{1}{s_{1}}, \cdots, \frac{1}{s_{m-1}}\big).
       \end{equation}

       We see that $p(r) + \frac{1}{\sqrt{N}}f_{N}(r)$ has a positive lower bound  on the  region of integration except for a neighbour of $s$ for sufficiently  large $N$. As in the proof of Theorem \eqref{laplaceMethod1} (see Condition \ref{condition3}  therein and Remark \ref{remark2.4}), the integral except for the part in that neighbour of $s$ is exponentially decaying. Furthermore, Laplace's approximation \cite[p.495]{wong2001asymptotic} can be used to give
       \begin{equation}
         \begin{split}
           w_{m}(z) &\sim \frac{2^{m-1} \abs{z}^{2(\sigma_{m}N + b_{m})}}{\pi} \prod_{k=1}^{m} \frac{\Gamma((\sigma_{k} + \tau_{k})N + b_{k} + 2)}{\Gamma(\sigma_{k} N + b_{k} + 1) \Gamma(\tau_{k}N + 1)}\\
           &\quad \times \Big(\frac{2\pi}{N}\Big)^{(m-1)/2} \big(\Det{A}\big)^{-1/2} q(s) e^{-Np(s) - \sqrt{N} f_{N}(s)} e^{\frac{1}{2} \eta A^{-1} \eta^{T}}
         \end{split}
       \end{equation}
       from which the desired result follows.
     \end{proof}

    We now  turn to consider the asymptotics of the truncated series given in (\ref{finitesumTr}). For $u\neq 0$, let $\zeta(u)$ be the unique positive solution of the algebraic equation
     \begin{equation}\label{algebraicEqTr2}
       \prod_{k=1}^{m} \frac{\sigma_{k} - \sigma_{m} + x\abs{u}^{2}}{\sigma_{k} - \sigma_{m} + \tau_{k} + x\abs{u}^{2}} - \abs{u}^{2} = 0.
     \end{equation}
     For  $1\leq k\leq m$ and $1\leq j\leq m-1$ set
     \begin{equation}\label{tildes-constantsTr}
      \tilde{s}_{k}  = \sigma_{k} - \sigma_{m} + \tau_{k} + \zeta(u)\abs{u}^{2},\  \tilde{s}_{m+j}  = \sigma_{j} - \sigma_{m} + \zeta(u)\abs{u}^{2},\ \tilde{s}_{2m}=1.\end{equation}
     Then we have the following lemma. Its proof is just the same as that of Lemma \ref{twopartslemmavarying} and thus may be omitted.
     \begin{lem}\label{twopartslemmavaryingTr}
       Suppose that $u$ is a fixed nonzero complex number and  $\tilde{s}_{k}$ is defined as in \eqref{tildes-constantsTr}.  Let   $\rho=\rho(u)\neq 0$ and introduce rescaling variables
       \begin{equation}
         z = u + \frac{v}{\rho \sqrt{N}}, \qquad z'= u + \frac{v'}{\rho \sqrt{N}}
       \end{equation}
       where $v$ and $v'$ lie in a compact set of $\Real{C}$, then for sufficiently large $N$ there exists $c_0>0$ such that
       \begin{align}
         T_{N}&(z, z')   = N^{2m + \sum_{k=1}^{m} \tau_{k}} (z\overline{z'})^{-(\sigma_{m}N + b_{m})} (2\pi i)^{-m} \prod_{k=1}^{m} \frac{\Gamma(\sigma_{k}N + b_{k} + 1)}{\Gamma(\sigma_{k} N + \tau_{k} N + b_{k} + 2)} \nonumber \\ & \quad  \times \Big(e^{-Np_{1}(\tilde{s})-c_0 N}  e^{-\sqrt{N} f(\tilde{s})}O(1) \nonumber\\
         &\hspace{2cm}+\int_{\gamma_{1} \times \cdots \times \gamma_{2m}}q(t) e^{-\sqrt{N} f(t)} \frac{e^{-Np_{1}(t)} - e^{-Np_{2}(t)}}{t_{2m}-1}\, d^{2m}t\Big) \label{decayandleadingvaryingTr}
       \end{align}
           where
           \begin{equation}\label{primativep1Tr}
             \begin{split}
               p_{1}(t) &= \sum_{k=1}^{m} (t_{k} - (\sigma_{k} - \sigma_{m} + \tau_{k}) \ln{t_{k}}) - \frac{t_{1} \cdots t_{m}}{t_{m+1} \cdots t_{2m-1}} \abs{u}^{2} t_{2m} \\
               &\quad - \sum_{k=1}^{m-1} (t_{m+k} - (\sigma_{k} - \sigma_{m}) \ln{t_{m+k}}) + \sigma_{m} \ln{t_{2m}},
             \end{split}
           \end{equation}
           \begin{equation}\label{primativep2Tr}
               p_{2}(t)  =  p_1(t) +   \ln{t_{2m}},
           \end{equation}
           \begin{equation}\label{primativeqTr}
             q(t) = \big(\prod_{k=1}^{m} t_{k}^{b_{k} - b_{m} + 1}\big) \big(\prod_{k=1}^{m-1} t_{m+k}^{-(b_{k} - b_{m} + 1)}\big) \, t_{2m}^{-b_{m}} e^{\frac{t_{1} \cdots t_{m}}{t_{m+1} \cdots t_{2m-1}} \frac{v\overline{v'}}{\rho^{2}} t_{2m}},
           \end{equation}
           and
           \begin{equation}\label{primativefTr}
             f(t) = -\frac{t_{1} \cdots t_{m}}{t_{m+1} \cdots t_{2m-1}} \frac{u\overline{v'} + v\overline{u}}{\rho} t_{2m}.
           \end{equation}
            Here    $\gamma_{k}=\{t_k: -\varepsilon_0 < t_k- \tilde{s}_{k} <\varepsilon_{0}\}$ and $\gamma_{m+k}=\{t_{m+k}:\abs{t_{m+k}}= \tilde{s}_{m+k}, |t_{m+k}- \tilde{s}_{m+k}|<\varepsilon_{0}\}$  for some $\varepsilon_{0}>0$ ($k=1,\ldots,m$). Moreover, both $\Re{p_{1}(t)}$ and $\Re{p_{2}(t)}$ attain their unique minimum over $\gamma_{1} \times \cdots \times \gamma_{2m}$ at the point $ \tilde{s}  = (\tilde{s}_{1}, \cdots, \tilde{s}_{m-1},\tilde{s}_{2m}).  $
     \end{lem}

     \begin{lem}\label{TrSumP1}
       With the same notation and assumptions as in Lemma \ref{twopartslemmavaryingTr}, let
       \begin{equation}
         \Upsilon = \mathrm{P.V.} \int_{\gamma_{1} \times \cdots \times \gamma_{2m}} q(t) e^{-\sqrt{N}f(t)} e^{-Np_{1}(t)} \frac{1}{t_{2m} - 1} d^{2m}t,
       \end{equation}
       then the following hold true uniformly for $v, v'$ in any compact subset of $\mathbb{C}$.
       \begin{enumerate}[label = $\mathrm{(\Roman*)}$]
         \item\label{TrStatement1} For $\abs{u} \neq \sqrt{\frac{\sigma_{1} \cdots \sigma_{m}}{(\sigma_{1} + \tau_{1}) \cdots (\sigma_{m} + \tau_{m})}}$,
             \begin{align*} \Upsilon &\sim \frac{- i^{m} \pi  (2\pi/N)^{(2m-1)/2}\, \mathrm{sign}(\sigma_{m} - \zeta(u)\abs{u}^{2})}{\pi \sqrt{ \big(1 - \zeta(u) \abs{u}^{2} (\sum_{k=1}^{m} \frac{1}{\tilde{s}_{k}} - \sum_{l=1}^{m-1} \frac{1}{\tilde{s}_{m+l}})\big)\prod_{k=1}^{2m-1} \frac{1}{\tilde{s}_{k}}}}\\
                 &\quad \times \mathrm{exp}\Big\{-N\sum_{k=1}^{m} \big(\tilde{s}_{k} - (\sigma_{k} - \sigma_{m} + \tau_{k}) \ln{\tilde{s}_{k}}\big) + N\zeta(u) \abs{u}^{2} \Big\}\\
                 &\quad \times \mathrm{exp}\Big\{N\sum_{k=1}^{m-1} \big(\tilde{s}_{m+k} - (\sigma_{k} - \sigma_{m}) \ln{\tilde{s}_{m+k}}) - N\sigma_{m}\Big\}\\
                 &\quad \times \mathrm{exp}\Big\{\sqrt{N} \zeta(u) \frac{u\overline{v'} + v\overline{u}}{\rho}\Big\} \ \prod_{k=1}^{m} \tilde{s}_{k}^{b_{k} - b_{m} + 1} \prod_{l=1}^{m-1} \tilde{s}_{m+l}^{-(b_{l} - b_{m} + 1)} \, e^{\zeta(u) \frac{v\overline{v'}}{\rho^{2}}} \\
                 &\quad \times \mathrm{exp}\Big\{\frac{1}{2} \zeta^{2}(u) \frac{(u\overline{v'} + v\overline{u})^{2}}{\rho^{2}} \frac{\sum_{k=1}^{m} \frac{1}{\tilde{s}_{k}} - \sum_{l=1}^{m-1} \frac{1}{\tilde{s}_{m+l}}}{1-\zeta(u)\abs{u}^{2}(\sum_{k=1}^{m} \frac{1}{\tilde{s}_{k}} - \sum_{l=1}^{m-1} \frac{1}{\tilde{s}_{m+l}})}\Big\}.
              \end{align*}
         \item\label{TrStatement2} For $\abs{u} = \sqrt{\frac{\sigma_{1} \cdots \sigma_{m}}{(\sigma_{1} + \tau_{1}) \cdots (\sigma_{m} + \tau_{m})}} \neq 0,$
             \begin{equation*}
               \begin{split}
                 \Upsilon &\sim \frac{i^{m} \pi  (2\pi/N)^{(2m-1)/2}}{\pi \sqrt{\big(1 - \zeta(u) \abs{u}^{2} (\sum_{k=1}^{m} \frac{1}{\tilde{s}_{k}} - \sum_{l=1}^{m-1} \frac{1}{\tilde{s}_{m+l}})\big)\prod_{k=1}^{2m-1} \frac{1}{\tilde{s}_{k}} }}\\
                 &\quad \times \mathrm{exp}\Big\{-N\sum_{k=1}^{m} \big(\tilde{s}_{k} - (\sigma_{k} - \sigma_{m} + \tau_{k}) \ln{\tilde{s}_{k}}\big) + N\zeta(u) \abs{u}^{2} \Big\}\\
                 &\quad \times \mathrm{exp}\Big\{N\sum_{k=1}^{m-1} \big(\tilde{s}_{m+k} - (\sigma_{k} - \sigma_{m}) \ln{\tilde{s}_{m+k}}) - N\sigma_{m}\Big\}\\
                 &\quad \times \mathrm{exp}\Big\{\sqrt{N} \zeta(u) \frac{u\overline{v'} + v\overline{u}}{\rho}\Big\} \prod_{k=1}^{m} \tilde{s}_{k}^{b_{k} - b_{m} + 1} \prod_{l=1}^{m-1} \tilde{s}_{m+l}^{-(b_{l} - b_{m} + 1)} e^{\zeta(u) \frac{v\overline{v'}}{\rho^{2}}} \\
                 &\quad \times \mathrm{exp}\Big\{\frac{1}{2} \zeta^{2}(u) \frac{(u\overline{v'} + v\overline{u})^{2}}{\rho^{2}} \frac{\sum_{k=1}^{m} \frac{1}{\tilde{s}_{k}} - \sum_{l=1}^{m-1} \frac{1}{\tilde{s}_{m+l}}}{1-\zeta(u)\abs{u}^{2}(\sum_{k=1}^{m} \frac{1}{\tilde{s}_{k}} - \sum_{l=1}^{m-1} \frac{1}{\tilde{s}_{m+l}})}\Big\}\\
                 &\quad \times \mathrm{erf}\Big(\frac{\sqrt{\zeta(u)}}{\sqrt{1-\zeta(u)\abs{u}^{2}(\sum_{k=1}^{m} \frac{1}{\tilde{s}_{k}} - \sum_{l=1}^{m-1} \frac{1}{\tilde{s}_{m+l}})}} \frac{u\overline{v'} + v\overline{u}}{\rho \sqrt{2} \abs{u}}\Big).
               \end{split}
             \end{equation*}
       \end{enumerate}
     \end{lem}
     \begin{proof}
      For  $\tilde{s}_{k}$ given in \eqref{tildes-constantsTr}, write  $\tilde{s} = (\tilde{s}_{1}, \cdots, \tilde{s}_{2m})$. Then we know that   the following inequality
       \begin{equation*}
         \Re{p_{1}(t) + \frac{1}{\sqrt{N}}f(t)} > \Re{p_{1}(\tilde{s}) + \frac{1}{\sqrt{N}}f(\tilde{s})}
       \end{equation*}
       holds for all $t \in \gamma_{1} \times \cdots \times \gamma_{2m}$ except for $\tilde{s}$.

       Notice
       \begin{equation*}
         \begin{split}
           \frac{\partial p_{1}(t)}{\partial t_{k}} &= 1 - (\sigma_{k} - \sigma_{m} + \tau_{k}) \frac{1}{t_{k}} - \frac{1}{t_{k}} \frac{t_{1} \cdots t_{m} \abs{u}^{2} t_{2m}}{t_{m+1} \cdots t_{2m-1}}, \quad 1\leq  k \leq  m;\\
           \frac{\partial p_{1}(t)}{\partial t_{m+k}} &= -1 + (\sigma_{k} - \sigma_{m}) \frac{1}{t_{k}} + \frac{1}{t_{k}} \frac{t_{1} \cdots t_{m} \abs{u}^{2} t_{2m}}{t_{m+1} \cdots t_{2m-1}}, \quad 1 \leq k< m;\\
           \frac{\partial p_{1}(t)}{\partial t_{2m}} &= - \frac{t_{1} \cdots t_{m} \abs{u}^{2}}{t_{m+1} \cdots t_{2m-1}} + \sigma_{m}\frac{1}{t_{2m}},
         \end{split}
       \end{equation*}
       we have
       \begin{equation*}
         \frac{\partial p_{1}(\tilde{s})}{\partial t_{k}} = 0,\quad k = 1, \cdots, 2m-1.
       \end{equation*}

       On the other hand, for the Hessian matrix
       \begin{equation}
         A = \paren{\frac{\partial^{2} p_{1}(\tilde{s})}{\partial t_{k} \partial t_{j}}}_{1 \leq k, j \leq 2m-1}
       \end{equation}
       where
       \begin{equation}
         \frac{\partial^{2} p_{1}(\tilde{s})}{\partial t_{k} \partial t_{j}} = \begin{cases}
           - \frac{1}{\tilde{s}_{k} \tilde{s}_{j}} \zeta(u) \abs{u}^{2}, & 1 \leq k \neq j \leq m \text{~or~} m < k \neq j < 2m;\\
           \frac{1}{\tilde{s}_{k} \tilde{s}_{j}} \zeta(u) \abs{u}^{2}, & 1 \leq k \leq m < j < 2m;\\
           \frac{1}{\tilde{s}_{k}} - \frac{1}{\tilde{s}_{k}^{2}} \zeta(u) \abs{u}^{2}, & 1 \leq k = j \leq m;\\
           -\frac{1}{\tilde{s}_{k}} - \frac{1}{\tilde{s}_{k}^{2}} \zeta(u) \abs{u}^{2}, & m < k = j < 2m,
         \end{cases}
       \end{equation}
       simple manipulation gives the determinant
       \begin{equation}
         \Det{A} = (-1)^{m-1} \Big(1 - \zeta(u) \abs{u}^{2} \big(\sum_{k=1}^{m} \frac{1}{\tilde{s}_{k}} - \sum_{l=1}^{m-1} \frac{1}{\tilde{s}_{m+l}}\big)\Big)\prod_{k=1}^{2m-1} \frac{1}{\tilde{s}_{k}}
       \end{equation}
       and the inverse
       \begin{multline}
         A^{-1} = \mathrm{diag}({\tilde{s}_{1}}, \cdots, {\tilde{s}_{m}}, -{\tilde{s}_{m+1}}, \cdots, -{\tilde{s}_{2m-1}}) \\
         + \frac{\zeta(u) \abs{u}^{2}}{ 1 - \zeta(u) \abs{u}^{2} \big(\sum_{k=1}^{m} \,\frac{1}{\tilde{s}_{k}} - \sum_{l=1}^{m-1} \frac{1}{\tilde{s}_{m+l}}\big)} 1_{(2m-1) \times (2m-1)},
       \end{multline}
       where $1_{(2m-1) \times (2m-1)}$ is a $(2m-1) \times (2m-1)$ matrix with all entries $1$.

       Write  $\eta = (\frac{\partial f(\tilde{s})}{\partial t_{1}}, \cdots, \frac{\partial f(\tilde{s})}{\partial t_{2m-1}})$  where
       \begin{equation}
         \frac{\partial f(\tilde{s})}{\partial t_{k}} = \begin{cases}
           -\zeta(u) \frac{u\overline{v'} + v\overline{u}}{\rho} \frac{1}{\tilde{s}_{k}}, &k = 1, \cdots, m, \\
           \zeta(u) \frac{u\overline{v'} + v\overline{u}}{\rho} \frac{1}{\tilde{s}_{k}}, &k = m+1, \cdots, 2m-1,
         \end{cases}
       \end{equation}
       then
       \begin{equation}
         \begin{split}
           \eta A^{-1} \eta^{T} &= \zeta^{2}(u) \frac{(u\overline{v'} + v\overline{u})^{2}}{\rho^{2}} \frac{\sum_{k=1}^{m} \frac{1}{\tilde{s}_{k}} - \sum_{l=1}^{m-1} \frac{1}{\tilde{s}_{m+l}}}{1-\zeta(u)\abs{u}^{2}(\sum_{k=1}^{m} \frac{1}{\tilde{s}_{k}} - \sum_{l=1}^{m-1} \frac{1}{\tilde{s}_{m+l}})}.
         \end{split}
       \end{equation}

       When $\abs{u} \neq \sqrt{\frac{\sigma_{1} \cdots \sigma_{m}}{(\sigma_{1} + \tau_{1}) \cdots (\sigma_{m} + \tau_{m})}}$, let
       \begin{equation}
         \alpha = \frac{\partial p_{1}(\tilde{s})}{\partial t_{2m}} = \sigma_{m} -\zeta(u) \abs{u}^{2},
       \end{equation}
       application of Theorem \ref{laplaceMethod1} gives
       \begin{equation}
         \begin{split}
           \Upsilon &\sim (-1)^{m} i\pi  (2\pi/\lambda)^{(2m-1)/2} q(\tilde{s}) e^{-\lambda p_{1}(\tilde{s})-\sqrt{\lambda} f(\tilde{s})} e^{\frac{1}{2} \eta A^{-1} \eta^{T}}\frac{\mathrm{sign}(\alpha)}{\sqrt{\det{A}}}
         \end{split}
       \end{equation}
       from which Statement \ref{TrStatement1} follows.

       If $\abs{u} = \sqrt{\frac{\sigma_{1} \cdots \sigma_{m}}{(\sigma_{1} + \tau_{1}) \cdots (\sigma_{m} + \tau_{m})}} \neq 0$, we then have
       \begin{equation}
         \frac{\partial p_{1}(\tilde{s})}{\partial t_{k}} = 0, \quad k = 1, \cdots, 2m.
       \end{equation}
       Write        \begin{equation}
         \beta = \zeta(u) \abs{u}^{2} \big(-\frac{1}{\tilde{s}_{1}}, \cdots, -\frac{1}{\tilde{s}_{m}}, \frac{1}{\tilde{s}_{m+1}}, \cdots, \frac{1}{\tilde{s}_{2m-1}}\big),
       \end{equation}
       and
       \begin{equation}
         \alpha' = \frac{\partial^{2} p_{1}(\tilde{s})}{\partial t_{2m}^{2}} = -\sigma_{m}, \qquad \eta_{2m} = \frac{\partial f(\tilde{s})}{\partial t_{2m}} = - \zeta(u) \frac{u\overline{v'} + v\overline{u}}{\rho}.
       \end{equation}
       By Theorem \ref{steepestDecentonEdge}, we obtain
       \begin{equation}
         \begin{split}
           \Upsilon &\sim (-1)^{m-1} (2\pi/\lambda)^{m/2} q(\tilde{s})  e^{-\lambda p(\tilde{s})-\sqrt{\lambda} f(\tilde{s})} e^{\frac{1}{2} \eta A^{-1} \eta^{T}} \frac{i\pi}{\sqrt{\det{A}}}\mathrm{erf}\big( \frac{
        i\eta_{2m} - i\beta A^{-1} \eta^{T}
        }{ \sqrt{2\alpha' - 2\beta A^{-1} \beta^{T}}
        }\big)
         \end{split}
       \end{equation}
       from which Statement \ref{TrStatement2} follows.

       Obviously, the asymptotics holds true uniformly for $v, v'$ in any compact subset of $\mathbb{C}$.
     \end{proof}

     Similar calculations as in the proof of Lemma \ref{TrSumP1} may afford us the asymptotics corresponding to the function $p_{2}(t)$.
     \begin{lem}\label{TrSumP2}
       With the same notation and assumptions as in Lemma \ref{twopartslemmavaryingTr}, let
       \begin{equation}
         \Phi = \mathrm{P.V.} \int_{\gamma_{1} \times \cdots \times \gamma_{2m}} q(t) e^{-\sqrt{N}f(t)} e^{-Np_{2}(t)} \frac{1}{t_{2m} - 1} d^{2m}t,
       \end{equation}
       then the following hold true uniformly for $v, v'$ in any compact subset of $\mathbb{C}$.
       \begin{enumerate}[label = $\mathrm{(\Roman*)}$]
         \item\label{TrStatement3} For $\abs{u} \neq \sqrt{\frac{(1+\sigma_{1}) \cdots (1+\sigma_{m})}{(1+\sigma_{1} + \tau_{1}) \cdots (1+\sigma_{m} + \tau_{m})}}$,
             \begin{equation*}
               \begin{split}
                 \Phi &\sim \frac{- i^{m} \pi  (2\pi/N)^{(2m-1)/2} \,\mathrm{sign}(1 + \sigma_{m} - \zeta(u)\abs{u}^{2})}{\pi \sqrt{\big(1 - \zeta(u) \abs{u}^{2} (\sum_{k=1}^{m} \frac{1}{\tilde{s}_{k}} - \sum_{l=1}^{m-1} \frac{1}{\tilde{s}_{m+l}})\big)\prod_{k=1}^{2m-1} \frac{1}{\tilde{s}_{k}} }}\\
                 &\quad \times \mathrm{exp}\Big\{-N\sum_{k=1}^{m} \big(\tilde{s}_{k} - (\sigma_{k} - \sigma_{m} + \tau_{k}) \ln{\tilde{s}_{k}}\big) + N\zeta(u) \abs{u}^{2} \Big\}\\
                 &\quad \times \mathrm{exp}\Big\{N\sum_{k=1}^{m-1} \big(\tilde{s}_{m+k} - (\sigma_{k} - \sigma_{m}) \ln{\tilde{s}_{m+k}}) - N\sigma_{m}\Big\}\\
                 &\quad \times \mathrm{exp}\Big\{\sqrt{N} \zeta(u) \frac{u\overline{v'} + v\overline{u}}{\rho}\Big\} \prod_{k=1}^{m} \tilde{s}_{k}^{b_{k} - b_{m} + 1} \prod_{l=1}^{m-1} \tilde{s}_{m+l}^{-(b_{l} - b_{m} + 1)} e^{\zeta(u) \frac{v\overline{v'}}{\rho^{2}}} \\
                 &\quad \times \mathrm{exp}\Big\{\frac{1}{2} \zeta^{2}(u) \frac{(u\overline{v'} + v\overline{u})^{2}}{\rho^{2}} \frac{\sum_{k=1}^{m} \frac{1}{\tilde{s}_{k}} - \sum_{l=1}^{m-1} \frac{1}{\tilde{s}_{m+l}}}{1-\zeta(u)\abs{u}^{2}(\sum_{k=1}^{m} \frac{1}{\tilde{s}_{k}} - \sum_{l=1}^{m-1} \frac{1}{\tilde{s}_{m+l}})}\Big\}
               \end{split}
             \end{equation*}
         \item\label{TrStatement4} For $\abs{u} = \sqrt{\frac{(1+\sigma_{1}) \cdots (1+\sigma_{m})}{(1+\sigma_{1} + \tau_{1}) \cdots (1+\sigma_{m} + \tau_{m})}},$
             \begin{equation*}
               \begin{split}
                 \Phi &\sim \frac{i^{m} \pi (2\pi/N)^{(2m-1)/2}}{\pi \sqrt{ \big(1 - \zeta(u) \abs{u}^{2} (\sum_{k=1}^{m} \frac{1}{\tilde{s}_{k}} - \sum_{k=1}^{m-1} \frac{1}{\tilde{s}_{m+k}})\big)\prod_{k=1}^{2m-1} \frac{1}{\tilde{s}_{k}}}}\\
                 &\quad \times \mathrm{exp}\Big\{-N\sum_{k=1}^{m} \big(\tilde{s}_{k} - (\sigma_{k} - \sigma_{m} + \tau_{k}) \ln{\tilde{s}_{k}}\big) + N\zeta(u) \abs{u}^{2} \Big\}\\
                 &\quad \times \mathrm{exp}\Big\{N\sum_{k=1}^{m-1} \big(\tilde{s}_{m+k} - (\sigma_{k} - \sigma_{m}) \ln{\tilde{s}_{m+k}}) - N\sigma_{m}\Big\}\\
                 &\quad \times \mathrm{exp}\Big\{\sqrt{N} \zeta(u) \frac{u\overline{v'} + v\overline{u}}{\rho}\Big\} \prod_{k=1}^{m} \tilde{s}_{k}^{b_{k} - b_{m} + 1} \prod_{l=1}^{m-1} \tilde{s}_{m+l}^{-(b_{l} - b_{m} + 1)} e^{\zeta(u) \frac{v\overline{v'}}{\rho^{2}}} \\
                 &\quad \times \mathrm{exp}\Big\{\frac{1}{2} \zeta^{2}(u) \frac{(u\overline{v'} + v\overline{u})^{2}}{\rho^{2}} \frac{\sum_{k=1}^{m} \frac{1}{\tilde{s}_{k}} - \sum_{l=1}^{m-1} \frac{1}{\tilde{s}_{m+l}}}{1-\zeta(u)\abs{u}^{2}(\sum_{k=1}^{m} \frac{1}{\tilde{s}_{k}} - \sum_{l=1}^{m-1} \frac{1}{\tilde{s}_{m+l}})}\Big\}\\
                 &\quad \times \mathrm{erf}\Big(\frac{\sqrt{\zeta(u)}}{\sqrt{1-\zeta(u)\abs{u}^{2}(\sum_{k=1}^{m} \frac{1}{\tilde{s}_{k}} - \sum_{l=1}^{m-1} \frac{1}{\tilde{s}_{m+l}})}} \frac{u\overline{v'} + v\overline{u}}{\rho \sqrt{2} \abs{u}}\Big).
               \end{split}
             \end{equation*}
       \end{enumerate}
     \end{lem}

    \subsection{Scaling limits}
      Now, we are ready to prove the limiting eigenvalue density and local universality for products of truncated unitary matrices.

      Notice that  when  $0<\abs{u} <1$   for  the solution $\xi(u)$ of the algebraic equation (\ref{algebraicEqTr1}) and the solution $\zeta(u)$ of  (\ref{algebraicEqTr2})  we have the following relation
      \begin{equation}
        \zeta(u) = \frac{\tau_{m}}{\xi(u) - \abs{u}^{2}}.
      \end{equation}
     In this case
      \begin{equation}
        s_{k} = \sqrt{\tilde{s}_{m+k}/\tilde{s}_{k}}, \quad k =1, \cdots, m-1,
      \end{equation}
      where $s_{k}$ and $\tilde{s}_{k}$ are  defined respectively by (\ref{s-constantsTr}) and (\ref{tildes-constantsTr}).

      In comparison with the limiting density  in \cite[eq.(3.12)]{akemann2014universal} where all $\sigma_{k}, b_{k}$ are zero and all $\tau_{k}$ are equal, we have a more general  result.
      \begin{thm} \label{limitdensityTr}
        For the weight function $w_{m}(z)$   given by (\ref{weightTr}), suppose that $$\sigma_{k}N + b_{k}>-1, \quad \sigma_{k}\geq 0 \ \ \mbox{and}  \quad \tau_{k}> 0 \ \ \mbox{for}\  k=1, \ldots, m.$$
        For $z\neq 0$, let $\zeta_{m}(z)$ be the largest real root of   algebraic  equation in $x$
        \begin{equation}
          \abs{z}^{2} \prod_{k=1}^{m} \big(\sigma_{k} - \sigma_{m} + \tau_{k} + x\abs{z}^{2}\big) - \prod_{k=1}^{m} \big(\sigma_{k} - \sigma_{m} + x\abs{z}^{2}\big) = 0.
        \end{equation} Then  the limiting  eigenvalue density
        \begin{equation}
          \begin{split}
            R_{1}(z) &:= \lim_{N \to \infty} \frac{1}{N}R_{N,1}(z)\\
            &= \frac{1}{\pi \abs{z}^{2}} \frac{1}{\sum_{k=1}^{m}\big(\frac{1}{ \sigma_{k} - \sigma_{m}   + \abs{z}^{2}\zeta(z)} - \frac{1}{ \sigma_{k} - \sigma_{m} + \tau_{k} + \abs{z}^{2}\zeta(z)}\big)} \\
            &  \times \chi_{\set{u:\sqrt{\frac{\sigma_{1} \cdots \sigma_{m}}{(\sigma_{1} + \tau_{1}) \cdots (\sigma_{m} + \tau_{m})}} < \abs{u} < \sqrt{\frac{(1+\sigma_{1}) \cdots (1+\sigma_{m})}{(1+\sigma_{1} + \tau_{1}) \cdots (1+\sigma_{m} + \tau_{m})}}}}(z)
          \end{split}
        \end{equation}
       holds true for any complex number $z$ such that   $0<\abs{z}<1$.
     \end{thm}
     \begin{proof}
       Taking $v = 0$ in Lemma \ref{TrWeight}, we have
         \begin{align} \label{weightTr2}
          & w_{m}(z)  \sim  \Big(\frac{2\pi}{N}\Big)^{(m-1)/2} \prod_{k=1}^{m} \frac{\Gamma((\sigma_{k} + \tau_{k})N + b_{k} + 2)}{\Gamma(\sigma_{k} N + b_{k} + 1) \Gamma(\tau_{k}N + 1)}\, \abs{z}^{2(\sigma_{m}N + b_{m})}   \nonumber  \\
           &  \times   \frac{1}{\pi \sqrt{ \Big(1 + \frac{\tau_{m} \abs{u}^{2} \xi(u)}{(\xi(u) - \abs{u}^{2})^{2}}  \sum_{k=1}^{m-1} \frac{(1-s_{k}^{2})^{2}}{\tau_{k} s_{k}^{2} }\Big)\prod_{k=1}^{m-1} \frac{\tau_{k}}{(1-s_{k}^{2})^{2}}}} \prod_{k=1}^{m-1} s_{k}^{2(b_{k} - b_{m})-1} \nonumber\\
           & \times \mathrm{exp}\Big\{ N\sum_{k=1}^{m-1} \big(2(\sigma_{k} - \sigma_{m})\ln{s_{k}} + \tau_{k} \ln(1-s_{k}^{2})\big)+N\tau_{m} \ln{\big(1-\frac{\abs{u}^{2}}{\xi(u)}\big)}\Big\}.
         \end{align}

       Taking $v, v' = 0$ in Lemmas \ref{twopartslemmavaryingTr}, \ref{TrSumP1} and \ref{TrSumP2}, we get the large $N$ asymptotics of (\ref{decayandleadingvaryingTr})
       \begin{align}\label{partialSumTr}
           T_{N}(z,z) &\sim N^{2m + N\sum_{k=1}^{m} \tau_{k}} \abs{z}^{-2(\sigma_{m}N + b_{m})} (2\pi)^{-m} \prod_{k=1}^{m} \frac{\Gamma(\sigma_{k}N + b_{k} + 1)}{\Gamma(\sigma_{k} N + \tau_{k} N + b_{k} + 2)} \nonumber\\
           & \times \frac{\pi  (2\pi/N)^{(2m-1)/2}}{\sqrt{\big(1 - \zeta(u) \abs{u}^{2} (\sum_{k=1}^{m} \frac{1}{\tilde{s}_{k}} - \sum_{l=1}^{m-1} \frac{1}{\tilde{s}_{m+l}})\big)\prod_{k=1}^{2m-1} \frac{1}{\tilde{s}_{k}} }}\nonumber\\
           &  \times \mathrm{exp}\Big\{-N\sum_{k=1}^{m} \big(\tilde{s}_{k} - (\sigma_{k} - \sigma_{m} + \tau_{k}) \ln{\tilde{s}_{k}}\big) + N\zeta(u) \abs{u}^{2}\Big\}\nonumber\\
           & \times \mathrm{exp}\Big\{N\sum_{k=1}^{m-1} \big(\tilde{s}_{m+k} - (\sigma_{k} - \sigma_{m}) \ln{\tilde{s}_{m+k}}\big)\Big\}\,\prod_{k=1}^{m} \tilde{s}_{k}^{b_{k} - b_{m} + 1} \prod_{l=1}^{m-1} \tilde{s}_{m+l}^{-(b_{l} - b_{m} + 1)}\nonumber\\
           &  \times \chi_{\set{y:\sqrt{\frac{\sigma_{1} \cdots \sigma_{m}}{(\sigma_{1} + \tau_{1}) \cdots (\sigma_{m} + \tau_{m})}} < \abs{y} < \sqrt{\frac{(1+\sigma_{1}) \cdots (1+\sigma_{m})}{(1+\sigma_{1} + \tau_{1}) \cdots (1+\sigma_{m} + \tau_{m})}}}}(u).
               \end{align}

       Noting $z = u$, since $R_{N,1}(z) = w_{m}(z) T_{N}(z,z)$, combination of (\ref{weightTr2}) and  (\ref{partialSumTr})  completes the proof.
     \end{proof}

     \begin{thm} \label{univeralitytheoremTr}
       With the same notation as in Theorem \ref{limitdensityTr}, let
       \begin{equation}
         \rho = \frac{1}{  \abs{z}} \frac{1}{ \sqrt{ \sum_{k=1}^{m}\big(\frac{1}{ \sigma_{k} - \sigma_{m}   + \abs{u}^{2}\zeta(u)} - \frac{1}{ \sigma_{k} - \sigma_{m} + \tau_{k} + \abs{u}^{2}\zeta(u)}\big)}}
       \end{equation}
       and introduce rescaling variables
       \begin{equation}
         z_{k} = u + \frac{v_{k}}{\rho \sqrt{N}}, \quad k=1, \cdots, n.
       \end{equation}
       Then  the following    hold true uniformly for  $v_{1}, \cdots, v_{n}$ in any compact subset of $\mathbb{C}$.
       \begin{enumerate} [label = $\mathrm{(\arabic*)}$]
         \item \textbf{Bulk limit}.   \label{TrScalingLimitBulk} For  $\sqrt{\frac{\sigma_{1} \cdots \sigma_{m}}{(\sigma_{1} + \tau_{1}) \cdots (\sigma_{m} + \tau_{m})}} < \abs{u} < \sqrt{\frac{(1+\sigma_{1}) \cdots (1+\sigma_{m})}{(1+\sigma_{1} + \tau_{1}) \cdots (1+\sigma_{m} + \tau_{m})}}$,
             \begin{multline}
               \lim_{N \to \infty} N^{-n} \rho^{-2n} R_{N,n}\big(z_1, \ldots,z_n\big) \\= \Det{\frac{1}{\pi}e^{-\frac{1}{2} (\abs{v_{k}}^{2} + \abs{v_{j}}^{2} - 2v_{k}\overline{v_{j}})}}_{1\leq k, j \leq n}.
             \end{multline}
         \item  \textbf{Inner edge}. For  $\delta_{1}, \ldots, \delta_{m}>0$ and $u = \sqrt{\frac{\sigma_{1} \cdots \sigma_{m}}{(\sigma_{1} + \tau_{1}) \cdots (\sigma_{m} + \tau_{m})}}e^{i\phi}$ with $0\leq \phi<2\pi$,
             \begin{multline}
               \lim_{N \to \infty}N^{-n} \rho^{-2n} R_{N,n}\big(z_1, \ldots,z_n\big) \\
               = \Det{\frac{1}{2\pi}e^{-\frac{1}{2} (\abs{v_{k}}^{2} + \abs{v_{j}}^{2} - 2v_{k}\overline{v_{j}})} \mathrm{erfc}(-\frac{e^{i\phi}\overline{v_{j}} + v_{k} e^{-i\phi}}{\sqrt{2} })}_{1\leq k, j \leq n}.
             \end{multline}
         \item  \textbf{Outer   edge}. For  $u = \sqrt{\frac{(1+\sigma_{1}) \cdots (1+\sigma_{m})}{(1+\sigma_{1} + \tau_{1}) \cdots (1+\sigma_{m} + \tau_{m})}} e^{i\phi}$ with $0\leq \phi<2\pi$,
             \begin{multline}
               \lim_{N \to \infty} N^{-n} \rho^{-2n} R_{N,n}\big(z_1, \ldots,z_n\big) \\
               = \Det{\frac{1}{2\pi}e^{-\frac{1}{2} (\abs{v_{k}}^{2} + \abs{v_{j}}^{2} - 2v_{k}\overline{v_{j}})} \mathrm{erfc}(\frac{e^{i\phi}\overline{v_{j}} + v_{k} e^{-i\phi}}{\sqrt{2} })}_{1\leq k, j \leq n}.
             \end{multline}
       \end{enumerate}
      \end{thm}
      \begin{proof}
        Let
        \begin{equation}
          \begin{split}
            \psi_{N}(v) &= \mathrm{exp}\Big\{\frac{1}{4}\zeta^{2}(u)\frac{(v\overline{u})^{2}-(u\overline{v})^{2}}{\rho^{2}} \frac{\sum_{k=1}^{m}\frac{1}{\tilde{s}_{k}} - \sum_{l=1}^{m-1}\frac{1}{\tilde{s}_{m+l}}}{1-\zeta(u)\abs{u}^{2}\big(\sum_{k=1}^{m}\frac{1}{\tilde{s}_{k}} - \sum_{l=1}^{m-1}\frac{1}{\tilde{s}_{m+l}}\big)}\Big\} \\
            &\quad \times \mathrm{exp}\Big\{\frac{1}{2}\sqrt{N}\zeta(u)\frac{v\overline{u}-u\overline{v}}{\rho} + \sigma_{m} \sqrt{N} \frac{u\overline{v} - v\overline{u}}{2\abs{u}^{2} \rho} -  \sigma_{m} \frac{(u\overline{v})^{2} - (v\overline{u})^{2}}{4\abs{u}^{4} \rho^{2}}\Big\}
          \end{split}
        \end{equation}
        and let $D = \mathrm{diag}(\psi_{N}(v_{1}), \cdots, \psi_{N}(v_{n}))$.

        In the bulk,    combining    Lemma \ref{TrWeight}, Lemma \ref{twopartslemmavaryingTr}, Statement \ref{TrStatement1} in Lemma \ref{TrSumP1} and  \ref{TrStatement3} in Lemma \ref{TrSumP2}, we obtain
        \begin{equation*}
               \psi_{N}(v_{k}) K_{N}(z_{k},z_{j}) \psi_{N}^{-1}(v_{j})          \sim N \frac{1}{\pi} \rho^{2}
            e^{-\frac{1}{2} (\abs{v_{k}}^{2} + \abs{v_{j}}^{2} - 2v_{k}\overline{v_{j}})}.
        \end{equation*}
        Furthermore,
        \begin{equation*}
          \begin{split}
            R_{N,n}\big(z_1, \ldots,z_n\big) &= \Det{D(K_{N}(z_{k}, z_{j}))D^{-1}}\\
            &\sim N^{n} {\rho^{2n}} \det\Big(\frac{1}{\pi} e^{-\frac{1}{2} (\abs{v_{k}}^{2} + \abs{v_{j}}^{2} - 2v_{k}\overline{v_{j}})}\Big)_{1 \leq k, j \leq n}.
          \end{split}
        \end{equation*}
        This completes the bulk limit  of the theorem.

        Likewise, combining    Lemmas \ref{TrWeight},   \ref{twopartslemmavaryingTr},  \ref{TrSumP1} and    \ref{TrSumP2} we can prove the inner and outer edge cases.
      \end{proof}

      \begin{rem}
        Although we assume that all $\tau_{k} > 0$ in this section, our conclusions remain valid under the condition that at least one of the $\tau_{k}$'s is positive. Indeed, in the latter case the asymptotics of truncated series remains true and   asymptotics of   weight functions almost corresponds to the case of those  $\tau_{k} > 0$.
      \end{rem}

\begin{acknow}
  The work of D.-Z.~Liu  was  supported by the National Natural Science Foundation of China (Grants 11301499 and 11171005), and by the Fundamental Research Funds for the Central Universities (Grant WK0010000048). The first named author  would like to thank Gernot Akemann for  very helpful discussions on the first draft.
\end{acknow}

\begin{appendix}   \section { Determinant and inverse of  one matrix} \label{appendixa}
When calculating the Hessian matrix and the inverse   we frequently encounter  one   type of matrix  like
\begin{align}
        A :=A(a_1,\ldots,a_n) =
        \begin{pmatrix}
          1 + a_{1}  &1& \cdots & 1\\
         1   &1 + a_2& \cdots & 1\\
         \vdots& \vdots & \ddots & \vdots \\
          1 &1& \cdots & 1 + a_n
        \end{pmatrix}.
      \end{align}
The determinant reads off \begin{equation}\Det{A}=a_{1}a_2\cdots a_n \Big(1+\sum_{k=1}^{n}\frac{1}{a_k}\Big). \end{equation} Moreover, if $a_{1}a_2\cdots a_n\neq 0$ and also $1+\sum_{k=1}^{n}\frac{1}{a_k}\neq 0$, then the inverse $A^{-1}$ is equal to

  \begin{align}
    \frac{1}{1+\displaystyle\sum_{k = 1}^{n} \frac{ 1}{a_{k} }}
          \begin{pmatrix}
          \frac{1}{a_1} + \displaystyle\sum_{k \neq 1} \frac{1}{a_1 a_{k} } &- \frac{1}{a_{1}a_{2}}& \cdots & - \frac{1}{a_{1}a_{n}}\\
         - \frac{1}{a_{2}a_{1}}   & \frac{1}{a_2} + \displaystyle\sum_{k \neq 2} \frac{ 1}{a_2 a_{k} }& \cdots & - \frac{1}{a_{2}a_{n}}\\
         \vdots& \vdots & \ddots & \vdots \\
         - \frac{1}{a_{n}a_{1}} &- \frac{1}{a_{n}a_{2}}& \cdots &  \frac{1}{a_n} + \displaystyle\sum_{k \neq n} \frac{ 1}{a_n a_{k}}
        \end{pmatrix}.
      \end{align}
Equivalently,
\begin{equation}
         A^{-1} = -\frac{1}{1+s}\mathrm{diag}\,\big( \frac{1}{a_1}, \cdots, \frac{1}{a_n}\big)  A\big(- a_1+  a_{1}s,\ldots,-a_n+  a_{n}s\big)\,\mathrm{diag}\big( \frac{1}{a_1}, \cdots, \frac{1}{a_n}\big)
       \end{equation}
where \begin{equation}
         s =  \sum_{k=1}^{n}\frac{1}{a_k}.
       \end{equation}

\end{appendix}


\end{document}